\documentclass[11pt]{article}
\usepackage[utf8]{inputenc}
\usepackage{amsmath}
\usepackage{hyperref}
\usepackage{amsfonts}
\usepackage{amsthm}
\usepackage{amssymb}
\usepackage{mathtools}
\usepackage[T1]{fontenc}
\usepackage{pst-node}
\usepackage{tikz}
\usepackage{tikz-cd}
\usepackage{caption}
\usepackage{bbm}
\usepackage{comment}
\usepackage{array}   
\newcolumntype{L}{>{$}l<{$}}
\usepackage{multirow}

\usetikzlibrary{calc,backgrounds}

\pgfdeclarelayer{background}
\pgfdeclarelayer{foreground}
\pgfsetlayers{background,main,foreground}

\usepackage{geometry}
\theoremstyle{plain}
\newtheorem{thm}{Theorem}[section]

\newtheorem{lem}[thm]{Lemma}
\newtheorem{prop}[thm]{Proposition}
\newtheorem{cor}[thm]{Corollary}

\newtheorem{mythm}{Theorem}

\theoremstyle{definition}
\newtheorem{defn}[thm]{Definition}

\theoremstyle{remark}
\newtheorem*{rem}{Remark}

\newcommand{\re}{\mathbb{R}}
\newcommand{\co}{\mathbb{C}}

\newcommand{\ze}{\mathbb{Z}}

\newcommand{\na}{\mathbb{N}}

\renewcommand{\Im}{\mathrm{Im} \, }

\newcommand{\Id}{\mathrm{Id}}

\newcommand{\diag}{\mathrm{diag}}

\newcommand{\ag}{\mathfrak{g}}
\newcommand{\ah}{\mathfrak{h}}

\newcommand{\al}{\mathfrak{l}}

\newcommand{\Cdim}{\mathrm{Confdim}}

\renewcommand{\a}{\alpha}

\newcommand{\g}{\gamma}
\newcommand{\G}{\Gamma}
\renewcommand{\d}{\delta}

\newcommand{\e}{\varepsilon}

\renewcommand{\l}{\lambda}
\renewcommand{\L}{\Lambda}
\newcommand{\s}{\sigma}

\newcommand{\Lrm}{\mathrm{L}}

\title{Quantitative polynomial cohomology and applications to $\textrm L^p$-measure equivalence}

\author{Antonio López Neumann and Juan Paucar}

\date{}

\begin{document}

\maketitle

\begin{abstract}

    We introduce a quantitative version of polynomial cohomology for discrete groups and show that it coincides with usual group cohomology when combinatorial filling functions are polynomially bounded. 
    As an application, we show that Betti numbers of nilpotent groups are invariant by mutually cobounded $\textrm L^p$-measure equivalence. We also use this to obtain new vanishing results for non-cocompact lattices in rank 1 simple Lie groups.

    \vspace*{2mm} \noindent{2020 Mathematics Subject Classification: }37A20, 20J06, 20F65, 20F18, 57M07, 22E41.


    \vspace*{2mm} \noindent{Keywords and Phrases:} Measure equivalence, nilpotent groups, group cohomology, Betti numbers, filling functions.
\end{abstract}

\tableofcontents

\section*{Introduction}

\textit{Measure equivalence} (\textit{ME}) was first introduced by Gromov as a measurable counterpart to quasi-isometry (QI) \cite[0.5.E]{gromov}. It simultaneously generalizes the situation of two lattices sitting inside a locally compact second countable group and that of orbit equivalence (OE). Some standard references include \cite{furman2011survey, gaboriau-survey}. For precise definitions, see Section \ref{Section: Preliminaries}.

Ornstein and Weiss showed that any two countable infinite amenable groups are OE, hence ME. This means in particular that both ME and OE do not take into account the geometry of these groups (as opposed to quasi-isometry). It is thus natural to refine the notion of ME, for instance by imposing integrability conditions.

    Conditions of this kind were first considered by Shalom, first in the form of $\mathrm{L}^2$ and $\textrm L^p$-integrability of lattices \cite{shalom2000rigidity-semisimple} and later in the form of uniform measure equivalence \cite{shalom2004harmonic}. Bader, Furman and Sauer properly defined $\textrm L^p$-\textit{measure equivalence} ({$\textrm L^p$-ME}, see Section \ref{Section: Preliminaries} for a definition) in \cite{bader-furman-sauer}, which has been more systematically studied and further generalized since then \cite{DKLMT}.

    One can see $\textrm L^p$-measure equivalence as a way of interpolating between measure equivalence and quasi-isometry (at least among amenable groups). Indeed, Shalom showed that two finitely generated amenable groups are quasi-isometric if and only if they are \textit{mutually cobounded} $\mathrm{L}^\infty$-ME \cite[2.1.7]{shalom2004harmonic}. Mutual coboundedness should be thought of as a fundamental domain compatibility condition that extends OE (see Section \ref{Section: Preliminaries} for a definition).
    
    A new goal would be to study infinite countable groups up to $\textrm L^p$-ME. An important result in this direction is due to Bowen, who showed that if two finitely generated groups are $\mathrm{L}^1$-ME, then they have the same asymptotic growth \cite[Appendix B]{austin2016}. This already shows the diversity of $\mathrm{L}^1$-ME classes among amenable groups, as it distinguishes virtually nilpotent groups, groups of intermediate growth, and amenable groups of exponential growth.
    
    


    \paragraph{On $\textrm L^p$-ME of nilpotent groups}
    One class of groups for which we can give more precise statements about their $\textrm L^p$-ME classification are (virtually) nilpotent groups. Indeed, Austin showed that finitely generated groups which are $\textrm L^1$-ME have bilipschitz asymptotic cones \cite[1.1]{austin2016}. Moreover, Pansu showed that all asymptotic cones of a simply connected nilpotent Lie group are isomorphic to its associated Carnot group \cite{pansu89-carnot}. Combining these two results, we obtain that two finitely generated nilpotent groups which are $\textrm L^1$-ME have the same associated Carnot group. On the other hand, Delabie, Llosa Isenrich and Tessera showed a (slightly stronger) converse to this result: two finitely generated virtually nilpotent groups with isomorphic associated Carnot groups are $\textrm L^p$-OE for some $p>1$ \cite{delabie-llosa-tessera}. We sum up these results in the following theorem.

    \begin{thm} \cite[1.5]{delabie-llosa-tessera}
    Let $\G$ and $\L$ be two finitely generated virtually nilpotent groups. The following are equivalent: \\
    $(1)$ $\G$ and $\L$ are $\mathrm{L^1}$-ME, \\
    $(2)$ $\G$ and $\L$ are $\mathrm L^p$-OE for some $p>1$, \\
    $(3)$ $\G$ and $\L$ have isomorphic associated Carnot groups.
\end{thm}

Given two nilpotent groups with the same associated Carnot, one may ask what is the largest $p >1$ for which they are $\textrm L^p$-OE or $\textrm L^p$-ME. A first obstruction to $\textrm L^p$-ME for large $p>1$ was obtained using a particular type of central extensions \cite[1.13]{delabie-llosa-tessera}.
Our first result gives an obstruction to mutually cobounded $\textrm L^p$-ME (and in particular $\textrm L^p$-OE) when the Betti numbers of the groups in question are different.

\begin{mythm}\label{Intro Theorem: invariance of Betti numbers of nilpotent groups} (Qualitative version)
    Let $\L$ and $\G$ be finitely generated virtually nilpotent groups and let $k \in \na$. There exists an explicit $p = p(\G, \L, k) < \infty$ such that if $\G$ and $\L$ are mutually cobounded $\textup L^p$-ME, then 
    $\dim_\re H^k(\L, \re) =\dim_\re H^k(\G, \re)$.
\end{mythm}

The number $p$ obtained in this Theorem is always $> 1 $, so in particular the groups $\L$ and $\G$ in question automatically have the same associated Carnot group. Moreover, this number can be described explicitly in terms of geometric/algebraic data of the groups; we will make this more precise later in this introduction (see Theorem \ref{Intro Theorem: Quantitative invariance of Betti numbers of nilpotent groups}). For instance, we may consider the simply connected nilpotent Lie group $L_n$ whose Lie algebra is given by
\begin{equation*}
    \al_n = \langle X_1, \ldots, X_n \, | \,  [X_i , X_j ] = (j-i)X_{i+j} \text{ for } i+j \leq n \rangle.
\end{equation*}
Its associated Carnot-graded nilpotent Lie group $\mathrm{gr}(L_n)$ has Lie algebra $\mathrm{gr}(\al_n)$, which is the so-called universal or model filiform algebra given by
\begin{equation*}
    \mathrm{gr}(\al_n) = \langle Y_1, \ldots, Y_n \, | \,  [Y_1 , Y_j ] = (j-1)Y_{j+1} \quad \text{ for } 2 \leq j \leq n -1 \rangle.
\end{equation*}
The quantitative version of Theorem \ref{Intro Theorem: invariance of Betti numbers of nilpotent groups} (that is Theorem \ref{Intro Theorem: Quantitative invariance of Betti numbers of nilpotent groups}) gives that the groups $L_n$ and $\mathrm{gr}(L_n)$ cannot be $\textrm L^p$-OE for $p > n^2 +n + 5$ when $ n \geq 7$. More examples are given in Section \ref{Section: Applications for virtually nilpotent groups}.

The following paragraphs of this introduction describe the tools used to prove Theorem \ref{Intro Theorem: invariance of Betti numbers of nilpotent groups}. The tools in question are induction of cohomology and a quantitative version of polynomial cohomology. The results shown for these tools (Theorems \ref{Intro Theorem: Injectivity of induction in polynomial cohomology} and \ref{Intro theorem: Quantitative isomorphism of polynomial cohom onto ordinary cohom}) yield a quantitative version of Theorem \ref{Intro Theorem: invariance of Betti numbers of nilpotent groups} (concretely stated as Theorem \ref{Intro Theorem: Quantitative invariance of Betti numbers of nilpotent groups}), but they are of independent interest and can be used in other contexts.

\paragraph{Invariance of cohomology and injectivity of induction} Betti numbers are not the first objects of cohomological nature to have been shown to remain invariant under some variant of measure equivalence. Here is a list of ME or $\textrm L^p$-ME invariants of a cohomological nature. \\
$\bullet$ (Furman) Property $(T)$ is invariant under ME \cite[1.4]{furman1999-ME}. \\
$\bullet$ (Gaboriau) $\ell^2$-Betti numbers of countable groups are invariant under ME (up to proportionality on the sequence) \cite[6.3]{gaboriau-invariantsl2}. \\
$\bullet$ (Shalom) Betti numbers of finitely generated nilpotent groups are invariant under quasi-isometry \cite[1.2]{shalom2004harmonic} (where quasi-isometry should be understood as mutually cobounded $\mathrm{L}^\infty$-ME, so we can see that Theorem \ref{Intro Theorem: invariance of Betti numbers of nilpotent groups} is in fact a strengthening of this result). \\
$\bullet$ (Sauer) The real cohomology ring of a finitely generated nilpotent group is invariant under quasi-isometry \cite[1.5]{sauer-QI-invariants}. \\
$\bullet$ (Monod-Shalom) Vanishing of $H_b^2 (\G, \pi)$ for every unitary representation $\pi$ is invariant under ME. Vanishing of $H^2_b(\G, \ell^2(\G))$ is invariant under ME \cite[7.6]{monod-shalom}. \\ 
$\bullet$ (Marrakchi-de la Salle) For $2< p <\infty$, the fixed point property $F \textrm L^p$ is invariant under $\textrm L^p$-ME \cite[7.4]{marrakchi-dlSalle}. \\
$\bullet$ (Das) For $1 \leq p < \infty$, vanishing of $\ell^p$-cohomology in degree 1 is invariant under $\textrm L^p$-ME \cite[1.1]{das-conformal}.

The general strategy of all of these results (including Shalom) is to show a version of Shapiro's lemma for measure equivalence. More precisely, given two ME groups $\L$ and $\G$ and some isometric representation of $\L$ on a Banach space $V$, we can use the $\G$-action on the $\L$-fundamental domain $X_\L$ of the ME-coupling to obtain an isometric representation $ I \rho$ of $\G$, which we call the \textit{induced representation} of $\rho$. There are many possible choices for the space on which the induced representation can act: to remain in the category of isometric representations on Banach spaces, we will always choose $L^p(X_\L, V)$ for some $1 \leq p < \infty$. The classical Shapiro lemma \cite[III.6.2]{brown-cohomology} states that if $\L$ is a finite index subgroup of $\G$, then for every $\L$-module $V$ there is an isomorphism 
\begin{equation*}
    H^*(\L, V) \xrightarrow[\simeq]{I} H^*(\G, I V).
\end{equation*}
In more general settings, in order to ensure that this induction map $I$ with values in $L^p(X_\L, V)$ is well defined at the level of cochains, some quantitative control is required. This control can come either from conditions on the coupling space (that is, sufficient integrability) or from imposing growth restrictions on the cochains. For instance, what Shalom really proved in \cite{shalom2004harmonic} to obtain invariance of Betti numbers of nilpotent groups, is that if $\L$ and $\G$ are two mutually cobounded $\mathrm{L}^\infty$-ME groups, then for every unitary representation $(\rho, V)$ of $\L$ we have continuous injective linear maps: 
\begin{align*}
    H^*(\L, \rho) \xhookrightarrow{I} H^*(\G, I \rho), \\
    \overline{H}^*(\L, \rho) \xhookrightarrow{I} \overline{H}^*(\G, I \rho).
\end{align*}
This result should be compared with the result by Monod and Shalom on bounded cohomology \cite[4.4]{monod-shalom}, stating that if $\L$ and $\G$ are ME groups, then for every unitary representation $(\rho, V)$ of $\L$ we have that:  \begin{equation*}
    H^2_b(\L, \rho) \xhookrightarrow[]{I} H^2_b(\G, I \rho).
\end{equation*}
Here no integrability condition is required on the ME, yet a somewhat similar result is obtained. The reason is that here, the quantitative control is placed on the cochains, which are bounded, and not on the integrability of the coupling. Another result of a similar nature is the one shown by Marrakchi and de la Salle \cite[7.4]{marrakchi-dlSalle}, stating that if $\L$ and $\G$ are finitely generated $\textrm L^p$-ME groups, then for every isometric representation $(\rho, V)$ on some uniformly convex Banach space $V$, the $\textrm L^p$-induction module $(I^p\rho, I^p(V))$ satisfies:
\begin{equation*}
    H^1(\L, \rho) \xhookrightarrow[]{I} H^1(\G, I^p\rho)
\end{equation*}
(it is unclear whether this holds for reduced cohomology). Here the conditions to ensure that the induction map is well-defined lie halfway between those of the last two mentioned results. Indeed, we require both the $\textrm L^p$-integrability condition on the ME (weaker than $\mathrm{L}^\infty$) and a condition on the cochains, since cocycles for degree 1 cohomology grow automatically at most at linear speed.

\paragraph{Quantitative polynomial cohomology and injectivity of induction}
These ideas were implicitly used by Bader and Sauer in \cite{bader-sauer}, who showed that by restricting to cochains that grow at most at polynomial speed and requiring $\textrm L^p$-integrability for every $p < \infty$, one obtains injectivity of the induction map from a lattice to its ambient locally compact group in \cite[6.23]{bader-sauer}.

In order to prove Theorem \ref{Intro Theorem: invariance of Betti numbers of nilpotent groups}, we show that the induction map $I$ is well-defined in the same way as Bader and Sauer, but for a quantitative variant of polynomial cohomology (where we keep track of the degrees of the polynomials) and in the more general setting of ME.

We now introduce our quantitative variant of polynomial cohomology.
Let $\G$ be a finitely generated group, with finite generating set $S$ and associated word length $| \cdot |$. Let $(\rho, V)$ be an isometric representation of $\G$. Fix $N \in \re_{\geq 0 }$. We define the space $C_{\mathrm{pol}\leq N}^k(\G, V)$ of \textit{ cochains of polynomial growth at most $N$} as the space of maps $c: \G^{k+1} \to V$ for which there exists a constant $M$ such that for every $\g_0 , \ldots, \g_k \in \G$ we have:
\begin{equation*}
    ||c(\g_0, \ldots, \g_k)|| \leq M (1 + |\g_0| + \ldots +|\g_k|)^N.
\end{equation*}
We endow $C_{\mathrm{pol}\leq N}^k(\G, V)$ with the topology of pointwise convergence.
The usual coboundary operator $d$ for group cohomology restricts to $\G$-equivariant cochains of polynomial growth at most $N$, yielding maps $d : C_{\mathrm{pol}\leq N}^k(\G, V)^\G \to C_{\mathrm{pol}\leq N}^{k+1}(\G, V)^\G$. 

\begin{defn}
    For $k \in \na$ and $M, N \in \re_{>0}$ with $M\geq N$, we define the \textit{$k$-th quantitative polynomial cohomology space of degrees $M$ and $N$} by:
    \begin{equation*}
        H_{\mathrm{pol}: M \to  N}^k(\G, V) : = \ker ( d |_{C_{\mathrm{pol}\leq N}^{k}(\G, V)^\G}) / \big(  d ( C_{\mathrm{pol}\leq M}^{k-1}(\G, V)^\G) \cap  C_{\mathrm{pol}\leq N}^{k}(\G, V)^\G \big),
        \end{equation*}
      Similarly, the \textit{$k$-th reduced quantitative polynomial cohomology space of degrees $M$ and $N$} is defined by:
        \begin{equation*}
        \overline{H}_{\mathrm{pol}: M \to  N}^k(G, V) : = \ker ( d |_{C_{\mathrm{pol}\leq N}^{k}(\G, V)^\G}) / \big( \overline{d ( C_{\mathrm{pol}\leq M}^{k-1}(\G, V)^\G)} \cap  C_{\mathrm{pol}\leq N}^{k}(\G, V)^\G \big).
    \end{equation*}
\end{defn}

The reason why we filter polynomial cochains using two constants $M, N$ in this definition is due to the fact that when taking the coboundary $dc$ of a cochain $c$, the polynomial growth of the image $dc$ may collapse and become much smaller than that of $c$ (the most extreme case being that of a cocycle, that is $dc = 0$). Hence a cocycle $b \in C_{\mathrm{pol}\leq N}^{k-1}(\G, V)$ could perfectly be a coboundary of some $c \in C_{\mathrm{pol}\leq M}^{k-1}(\G, V)$ for some $M\gg N$, while not being the coboundary of any $c \in C_{\mathrm{pol}\leq N}^k(\G, V)$.
For more details on this definition, we refer to Section \ref{Section: A quantitative version of polynomial cohomology}.

Our main statement regarding induction of polynomial cohomology is the following (which can be stated in more general settings, but under more restrictive and complicated conditions, see Proposition \ref{Induction on quantitative polynomial cohomology} and Theorem \ref{Higher transfer operator is well defined}).

\begin{mythm}\label{Intro Theorem: Injectivity of induction in polynomial cohomology} Let $k \in \na_{\geq 2}$ and $M, N \in \re_{>0}$, with $M\geq N$. Let $\L$ and $\G$ be two finitely generated virtually nilpotent groups with ME coupling $(\Omega, m, X_\L, X_\G)$ and let $d(\G)$ denote the degree of polynomial growth of $\G$.

\begin{enumerate}
    \item[(1)] Induction \emph{(Proposition \ref{Induction on quantitative polynomial cohomology})}. Suppose that $(\Omega, m, X_\L, X_\G)$ is $L^{2M}$-integrable on $\L$. Then the induction maps $ I = I_j : C_{\mathrm{pol} \leq M}^j(\L, V)^\L \to C_{\mathrm{pol} \leq M}^j(\G, L^2(X_\L, V))^\G$ are well-defined linear continuous maps that commute with differentials for $j \geq 0$.
    \item[(2)] Transfer \emph{(Theorem \ref{Higher transfer operator is well defined})}. On top of that, suppose that $X_\G \subseteq X_\L$ and that $(\Omega, m, X_\L, X_\G)$ is $L^{r}$-integrable on $\G$ for some $r >  d(\G) k + 2M+1$. Then there exist continuous linear maps $T = T_j : C^j_{\mathrm{pol} \leq M} (\G,\textrm L^2(X_\L, V))^\G \to C^j(\L , V)^\L$ for $j \leq k$ (that we call \emph{transfer maps}) that commute with differentials and satisfy $T \circ I = \Id$. The maps $I$ and $T$ induce continuous linear maps in cohomology and reduced cohomology, such that the compositions
\begin{align*}
    &   H_{\mathrm{pol}:M \to N}^k(\L , \rho) \xrightarrow{I} H^k_{\mathrm{pol}:M \to N} (\G, I^2(\rho)) \xrightarrow{T}  H^k(\L , \rho) , \\
    &  \overline{H}_{\mathrm{pol}:M \to N}^k(\L , \rho) \xrightarrow{I} \overline{H}^k_{\mathrm{pol}:M \to N} (\G,I^2(\rho)) \xrightarrow{T}  \overline{H}^k(\L , \rho) .
\end{align*}
are the comparison maps induced by the inclusion $ C_{\mathrm{pol} \leq M}^*(\L, V) \xhookrightarrow{}  C^*(\L, V)$.
\end{enumerate}
\end{mythm}

The main technical point in the proof of Theorem \ref{Intro Theorem: Injectivity of induction in polynomial cohomology} is item $(2)$, and in particular, it is to show that the transfer maps $T$, defined in the same way as in \cite[3.2.1]{shalom2004harmonic}, are well-defined in our setting.

Recall that the number $d(\G)$ in Theorem \ref{Intro Theorem: Injectivity of induction in polynomial cohomology} can be computed explicitly from the central series $C^j(\G) = [\G, C^{j-1}(\G)]$ of $\G$ using the Bass-Guivarc'h formula
\begin{equation*}
    d(\G) = \sum_{j= 1}^s j \, \mathrm{rk}(C^{j-1}(\G) / C^j(\G)).
\end{equation*}
Since Theorem \ref{Intro Theorem: Injectivity of induction in polynomial cohomology} is mostly interesting for nilpotent groups $\L$ and $\G$ with the same associated Carnot group, they are at least $L^1$-ME, so in particular $d(\L) = d(\G)$ (combining \cite[1.5]{delabie-llosa-tessera} and \cite[Appendix B]{austin2016}).

\paragraph{Comparing quantitative polynomial cohomology and usual cohomology}
The next ingredient for proving Theorem \ref{Intro Theorem: invariance of Betti numbers of nilpotent groups} is understanding when the comparison map appearing in Theorem \ref{Intro Theorem: Injectivity of induction in polynomial cohomology} is an isomorphism.
This is exactly the content of the work of Bader and Sauer (originally due to Ji and Ramsey \cite{ji-ramsey}): a comparison theorem between polynomial cohomology and ordinary cohomology \cite[6.14]{bader-sauer}.
Assuming a polynomial control on combinatorial versions of the filling functions $\mathrm{cFV}_\G^j$ of a group $\G$ (for precise definitions, see Section \ref{Section: Preliminaries}), we establish a comparison theorem between quantitative polynomial cohomology and usual cohomology, depending on the best polynomial bounds that the functions $\mathrm{cFV}_\G^j$ satisfy. We denote:
\begin{equation*}
    \deg \mathrm{cFV}^j_\G := \inf \{ d \in \re_{>0},  \mathrm{cFV}_\G^j(t) \lesssim t^d \} \in [1 , + \infty].
\end{equation*}

\begin{mythm} \label{Intro theorem: Quantitative isomorphism of polynomial cohom onto ordinary cohom} 
Let $\G$ be a group of type $F_{d+1}$ for some $d \in \na_{\geq 1}$. Suppose that all filling functions $ \mathrm{cFV}_\G^i$ are polynomially bounded for $2 \leq i \leq d+1$. Let $V$ be an isometric representation of $\G$ on some Banach space $V$. For $1 \leq k \leq d$, there exists a constant $N_k>0$ such that for all $N > N_k$ and $M \geq N$ the comparison map:
 \begin{equation*}
     H_{\mathrm{pol}: M \to N}^k(\G, V) \to H^k(\G, V)
 \end{equation*}
is surjective (the same holds for their respective reduced versions). Moreover, there exist functions $\a_j: \re_{>0} \to \re$ for $1 \leq j \leq d$ such that for all $N > N_k$ and $M > \a_{k-1}(N)$ the comparison map:
 \begin{equation*}
     H_{\mathrm{pol}: M \to N}^k(\G, V) \to H^k(\G, V)
 \end{equation*}
is a continuous isomorphism of topological vector spaces (the same holds for their respective reduced versions). We can choose $N_1 = 1$ and for $k \geq 2$, 
\begin{equation*}
    N_k = \prod_{j= 2}^k \mathrm{deg}( \mathrm{cFV}_\G^j).
\end{equation*} The functions $\a_i$ are defined inductively by $\a_1(N) = N+1$ and for $i \geq 2$, we have:
    \begin{equation*}
        \a_i(N) = \max \{ (N+1) \prod_{j= 2}^i  \mathrm{deg}( \mathrm{cFV}_\G^j),  \prod_{j= 2}^i  \mathrm{deg}( \mathrm{cFV}_\G^j) + \a_{i-1}(N) \}.
    \end{equation*}
If the filling functions $ \mathrm{cFV}_\G^j$ are polynomial for $2 \leq j \leq k$, then surjectivity of the comparison map holds for $N = N_k$ and isomorphism holds for $N = N_k$ and $M = \a_{k-1}(N) = \a_{k-1}(N_k)$. 
\end{mythm}

The proof of Theorem \ref{Intro theorem: Quantitative isomorphism of polynomial cohom onto ordinary cohom}, just as in the work of Bader and Sauer, consists in revisiting the classical proof showing that group cohomology can be computed using equivariant cohomology of a simplicial complex on which $\G$ acts geometrically, but this time we keep track of the degrees of the polynomials each time we take a filling. Homological algebra is technically more demanding this time, as keeping track of the degrees of the polynomials forces us to see spaces as modules over a different algebra each time the degree of the polynomial increases.

Putting Theorems \ref{Intro Theorem: Injectivity of induction in polynomial cohomology} and \ref{Intro theorem: Quantitative isomorphism of polynomial cohom onto ordinary cohom} together we can prove Theorem \ref{Intro Theorem: invariance of Betti numbers of nilpotent groups} (with the same strategy as in \cite[4.1.1]{shalom2004harmonic}) and keep track of the constant $p = p (\G, \L, k)$.

\begin{mythm}\label{Intro Theorem: Quantitative invariance of Betti numbers of nilpotent groups} (Quantitative version of Theorem \ref{Intro Theorem: invariance of Betti numbers of nilpotent groups})
Let $\L$ and $\G$ be finitely generated virtually nilpotent groups with the same associated Carnot group. Let $d$ be the common degree of polynomial growth of $\L$ and $\G$ and $k \in \na_{\geq 2}$. \\ Let $N_k = \max \{ \prod_{i=2}^k \deg \mathrm{cFV_\L^i} ,\prod_{i=2}^k \deg \mathrm{cFV_\G^i}\}$. \\
$\bullet$ Suppose that $\L$ and $\G$ admit a mutually cobounded $\textrm L^p$-ME coupling for some \begin{equation*}
    p > 2  d + 2N_2 + 3.
\end{equation*}
Then $\dim_\re H^2(\L, \re) = \dim_\re H^2(\G, \re)$. \\
$\bullet$ Suppose that $\L$ and $\G$ admit a mutually cobounded $\textrm L^p$-ME coupling for some \begin{equation*}
    p >  k d+ 2 (N_k^2 + (k-2) N_k) + 1.
\end{equation*}
Then $\dim_\re H^k(\L, \re) = \dim_\re H^k(\G, \re)$.
\end{mythm}

\paragraph{Other applications: lattices in rank 1}
It is worth noting that Theorem \ref{Intro theorem: Quantitative isomorphism of polynomial cohom onto ordinary cohom} is more versatile and can give results in other contexts. For example, we can complete some of the results obtained by Bader and Sauer on cohomology of non-cocompact lattices in simple Lie groups \cite{bader-sauer}. Since their work uses polynomial cohomology without controlling the degrees of its polynomials, they can only deal with lattices in higher-rank semisimple groups as these are $\textrm L^p$-integrable for every $p<\infty$. A non-cocompact lattice $\L$ in a rank 1 simple Lie group $G$ is only $\textrm L^p$-integrable for $p < \mathrm{Confdim}(\partial G)$ \cite{shalom2000rigidity-semisimple}. We prove the following theorem, which combined with upper bounds for filling functions of lattices in rank 1 simple Lie groups, can be used to obtain statements that are similar to those of Bader and Sauer.

\begin{mythm} \label{Intro Theorem: Surjectivity from induced cohomology of rank 1 lattices} 
Let $\L$ be a non-cocompact lattice in a simple Lie group $G$ of rank 1. Let $k \geq 1$ and suppose that the filling functions $\mathrm{cFV}^j_\L$ are polynomially bounded for $j \leq k$ and let $N_1 = 1$ and $N_k = \prod_{i= 2}^k \deg \mathrm{cFV}_\L^i$. For every 
\begin{equation*}
    1 \leq p < \frac{\Cdim (\partial\mathbb{H}^n_\mathbb{K})}{N_k},
\end{equation*}
and for every isometric representation $\rho$ of $\L$ on some Banach space $V$, there exist continuous surjective linear maps:
    \begin{align*}
        H^k(G, I^p(\rho)) \twoheadrightarrow H^k(\L, \rho), \\
        \overline{H}^k(G, I^p(\rho)) \twoheadrightarrow \overline{H}^k(\L, \rho).
    \end{align*}
\end{mythm}

For instance, Theorem \ref{Intro Theorem: Surjectivity from induced cohomology of rank 1 lattices} combined with Wenger's upper bounds for filling functions of Carnot groups \cite[7.3]{wenger-asymptotic-rank} imply that non-cocompact lattices in $G = \mathrm{Isom}_0 (\mathbb{H}^2_\mathbb{O})$ inherit property $[T_3]$ \cite[Definition 1.1]{bader-sauer} from the group $G$.

\begin{cor} \label{Intro Corollary: Octonionic lattices have T_3} (Corollary \ref{Octonionic lattices have T_3} in the text)
    Let $\L$ be a non-cocompact lattice in $G = \mathrm{Isom}_0 (\mathbb{H}^2_\mathbb{O})$. Then $\L$ has property $[T_3]$, that is, for every unitary representation $\pi$ of $\L$, we have $H^k (\L, \pi) = 0$ for $k =1, 2, 3$.
\end{cor}

This Corollary allows to enlarge the construction of simple Kazhdan groups given in \cite{fournier-facio-sauer} by allowing to take non-uniform lattices as a starting point.

Theorem \ref{Intro Theorem: Surjectivity from induced cohomology of rank 1 lattices} combined with sharper estimates by Leuzinger \cite{leuzinger-rank1} and Gruber \cite{gruber-filling} for filling functions of rank 1 lattices imply that Pansu's vanishing results of $\textrm L^p$-cohomology of pinched negatively curved manifolds \cite{pansu08} still hold for non-cocompact lattices in rank 1 simple Lie groups.

\begin{cor} (Corollary \ref{Vanishing of Lp-cohom for rank 1 lattices} in the text) \label{Intro Theorem: Lp cohomology of rank 1 lattices}
Let $\L$ be a non-cocompact lattice in $G = \mathrm{Isom}_0 (\mathbb{H}^n_\mathbb{K})$, where $\mathbb{H}^n_\mathbb{K}$ denotes the hyperbolic space of dimension $n$ over $\mathbb{K} = \re, \co$ or the quaternions $\mathbb{H}$. Let $k \geq 1$ and $1 \leq p < \infty$. We have $H^k(\L, \ell^p(\L)) = 0$ if \\
$\bullet$ $\mathbb{K} = \re$, $1 \leq k \leq n-2$ and $1 \leq p <  \frac{n-1}{k}$, \\
$\bullet $ $\mathbb{K}= \co$, $1 \leq k < 2n-2$ and $1 \leq p <  \frac{2n + k -1}{2k}$, \\
$\bullet $ $\mathbb{K} =  \mathbb{H}$, $1 \leq k \leq n-1$ and $1 \leq p <  \frac{4n + k -1}{2k}$, \\
$\bullet $ $\mathbb{K} =  \mathbb{O}$, $k = 1, 2, 3$ and $1 \leq p <  \frac{15 + k}{2k}$.
\end{cor}

We remark that to prove Theorem \ref{Intro Theorem: Surjectivity from induced cohomology of rank 1 lattices} and its consequences we only use the surjectivity statement of Theorem \ref{Intro theorem: Quantitative isomorphism of polynomial cohom onto ordinary cohom} (whose proof is much easier than that of the isomorphism statement).

\paragraph{Outline of the paper}
Section \ref{Section: Preliminaries} describes the general setting of ME and $\textrm L^p$-ME and introduces combinatorial versions of filling functions. Section \ref{Section: A quantitative version of polynomial cohomology} defines quantitative polynomial cohomology and introduces other technical tools (like polynomially decaying algebras) to prove Theorem \ref{Intro theorem: Quantitative isomorphism of polynomial cohom onto ordinary cohom}. Section \ref{Section: Induction and transfer} explains how to induce representations through an ME coupling, how to induce polynomial cohomology, defines the transfer operator $T$ on polynomial cohomology and shows Theorem \ref{Intro Theorem: Injectivity of induction in polynomial cohomology} (without the nilpotence assumptions). Section \ref{Section: Applications for virtually nilpotent groups} explains how to use Sections \ref{Section: A quantitative version of polynomial cohomology} and \ref{Section: Induction and transfer} to deduce Theorem \ref{Intro Theorem: Quantitative invariance of Betti numbers of nilpotent groups} (and hence Theorem \ref{Intro Theorem: invariance of Betti numbers of nilpotent groups} as well). It also gives more concrete examples of families of nilpotent groups that we can distinguish up to (mutually cobounded) $\textrm L^p$-ME using Theorem \ref{Intro Theorem: Quantitative invariance of Betti numbers of nilpotent groups}. Section \ref{Section: Applications to non-cocomp lattices in rank 1} explains how to use Theorem \ref{Intro theorem: Quantitative isomorphism of polynomial cohom onto ordinary cohom} to obtain Theorem \ref{Intro Theorem: Surjectivity from induced cohomology of rank 1 lattices} and its corollaries.

\paragraph{Acknowledgements} The first author was funded by the National Science Center Grant Maestro-13 UMO-2021/42/A/ST1/00306 and by Fondation Sciences Mathématiques de Paris. He wants to thank Piotr Nowak for many useful exchanges. Both authors thank Claudio Llosa Isenrich and Romain Tessera for their interest, as well as discussions and intuitions they shared with us. We thank Francesco Fournier-Facio and Roman Sauer for pointing us to their article \cite{fournier-facio-sauer}, which motivated us to include a proof of Corollary \ref{Intro Corollary: Octonionic lattices have T_3} in this new version.
Finally, we are indebted to Gabriel Pallier for his generous help and precise references regarding polynomiality of filling functions of nilpotent groups.

\section{Preliminaries} \label{Section: Preliminaries}

\subsection{Measure equivalence}

\begin{defn} (Gromov)
    Let $\L$ and $\G$ be two countable discrete groups. We say that $\L$ and $\G$ are \textit{measure equivalent} (\textit{ME}) if there exist commuting actions of $\L$ and $\G$ on some standard measure space $(\Omega, m)$ that are measure preserving, free and with finite measure Borel fundamental domains $X_\L$ and $X_\G$ respectively. The space $(\Omega, X_\L, X_\G, m)$ is called a \textit{ME-coupling} for $\L$ and $\G$.
\end{defn}

Let $\L$ and $\G$ be two measure equivalent countable groups.

In what follows we fix a ME-coupling $(\Omega, X_\L, X_\G, m)$ between $\L$ and $\G$, as in the definition of measure equivalence. The fundamental domain $X_\G$ determines an  $\G$-equivariant map $p_\G : \Omega \to \G$ such that $p_\G (x) = e_\G \iff x \in X_\G$ and $p_\G(\g . x) = \g \, p_\G(x)$ for all $\g \in \G$ and almost every $x \in \Omega$. This map satisfies $p_\G(y)^{-1} . y \in X_\G$ for almost every $y \in \Omega$. Similarly, the choice of the fundamental domain $X_\L$ determines a $\L$-equivariant map $p_\L : \Omega \to \L$ such that $p_\L (x) = e_\L \iff x \in X_\L$, $p_\L(\l . x) = \l \, p_\L(x)$ and $p_\L(x)^{-1} .x \in X_\L$ for all $\l \in \L$ and almost every $x \in \Omega$.

Since the actions of $\L$ and $\G$ on $\Omega$ commute, we can define the induced action of $\G$ on $\Omega / \L$ (resp. the induced action of $\L$ on $\Omega / \G$). We will identify these quotient spaces with the corresponding $\L$ and $\G$-fundamental domains. In order to understand the induced actions under this identification, we define cocycles. \newline


\begin{defn}
We define the \textit{cocycles} $c_\L$ and $ c_\G$ associated to the ME-coupling $(\Omega, X_\L, X_\G, m)$ by
\begin{align*}
    c_\G : \L \times X_\G & \to \G \qquad & c_\L : \G \times X_\L &\to \L \\
    (\l, x)&\mapsto p_\G(\l.x)^{-1} & (\g, x) & \mapsto p_\L(\g.x)^{-1}
\end{align*}
\end{defn}

Using these cocycles we obtain an induced action of $\G$ on the fundamental domain $X_\L$. Indeed, let $\g \in \G$,  $x \in X_\L \mapsto \g \cdot x : =  c_\L(\g, x). \g . x \in X_\L$. Notice that this induced action preserves the finite measure $m|_{X_\L}$. Similarly, we obtain an induced action of $\L$ on $X_\G$ defined by: let $\l \in \L$, $x \in X_\G \mapsto \l \cdot x : =  c_\G(\l, x) . \l . x \in X_\G$, which preserves the measure $m|_{X_\G}$.

These cocycles satisfy the so-called \textit{cocycle relation}. Let $\g_1, \g_2 \in \G, \l_1, \l_2 \in \L$ and $x \in X_\L, y \in X_\G$, we have:
\begin{align*}
    c_\L(\g_1 \g_2, x) &= c_\L(\g_1, \g_2 \cdot x) c_\L(\g_2, x), \\
    c_\G(\l_1 \l_2, y) & = c_\G(\l_1, \l_2 \cdot y) c_\G(\l_2, y). 
\end{align*}
In practice, it will be more useful to work with $\a_\L (\g, x) := c_\L (\g^{-1}, x)^{-1}$ and $\a_\G (\l, y) := c_\G (\l^{-1}, y)^{-1}$. The cocycle relations for $\a_\L$ and for $\a_\G$ are:
\begin{align*}
    \a_\L(\g_1 \g_2, x) = \a_\L (\g_1, x) \a_\L (\g_2, \g_1^{-1} \cdot x), \\
    \a_\G(\l_1 \l_2, y) = \a_\G (\l_1, y) \a_\G (\l_2, \l_1^{-1} \cdot y), 
\end{align*}
where $\g_1, \g_2 \in \G, \l_1, \l_2 \in \L$ and $x \in X_\L, y \in X_\G$.

\paragraph{$\textrm L^p$-measure equivalence}

We now define $\textrm L^p$-measure equivalence, by requiring cocycles to be $\textrm L^p$.

\begin{defn}
    Let $(\G, S_\G)$ and $(\L, S_\L)$ be two finitely generated measure equivalent groups, with ME-coupling $(\Omega, X_\G, X_\L,  m)$ and associated cocycles $c_\G : \L \times X_\G \to \G$ and $c_\L : \G \times X_\L \to \L$. Denote by $|\cdot|_{S_\G}$ and $|\cdot|_{S_\L}$ the word metrics associated to the symmetric generating sets $S_\G$ and $S_\L$. Let $p > 0$. We say that the ME-coupling is an \textit{$\textrm L^p$-ME coupling on $\L$} if for all $\g \in \G$ we have:
    \begin{equation*}
        \int_{X_\L} |\a_\L(\g, x)|_{S_\L}^p \, \mathrm{d}m(x) < + \infty.
    \end{equation*}
    Similarly, we say that the ME-coupling is an \textit{$\textrm L^p$-ME coupling on $\G$} if for all $\l \in \L$ we have:
    \begin{equation*}
        \int_{X_\G} |\a_\G(\l, x)|_{S_\G}^p \, \mathrm{d}m(x) < + \infty.
    \end{equation*}
\end{defn}

We will now state a standard lemma that follows from the cocycle relation. It allows us (among other things) to test $\textrm L^p$-ME only on generating sets.

\begin{lem}\label{Growth of integrals is sublinear}
    Let $(\G, S_\G)$ and $(\L, S_\L)$ be two finitely generated ME groups and $(\Omega, m)$ a ME-coupling with fundamental domains $X_\L$ and $X_\G$ and cocycles $c_\L$ and $c_\G$. For $p>0$ define 
    \begin{equation*}
        M_p = \sup_{s \in S_\G} \Big( \int_{X_\L} |\a_\L(s, x)|_{S_\L}^p \, \mathrm{d}m(x) \Big)^{1/p}.
    \end{equation*}
    Then for every $p >0$ and $\g \in \G$ we have:

    \begin{equation*}
        \Big( \int_{X_\L} |\a_\L(\g, x)|_{S_\L}^p \, \mathrm{d}m(x) \Big)^{1/p} \leq M_p | \g |_{S_\G}.
    \end{equation*}
\end{lem}
\begin{proof}
    Let $\g_1, \g_2 \in \G$. We successively integrate the cocycle relation on $x \in X_\L$ and use the triangle inequality and Minkowski's inequality to obtain:
    \begin{align*}
         & \Big( \int_{X_\L} |\a_\L(\g_1 \g_2, x)|_{S_\L}^p \, \mathrm{d} m (x) \Big)^{1/p} \\ 
         &  =  \Big( \int_{X_\L} |\a_\L (\g_1, x) \a_\L (\g_2, \g_1^{-1} \cdot x)|_{S_\L}^p \, \mathrm{d} m (x) \Big)^{1/p} \\
         & \leq  \Big( \int_{X_\L} \big(|\a_\L (\g_1, x)|_{S_\L} + |\a_\L (\g_2, \g_1^{-1} \cdot x)|_{S_\L} \big)^p \, \mathrm{d} m (x) \Big)^{1/p} \\
         & \leq  \Big( \int_{X_\L} |\a_\L (\g_1, x) |_{S_\L}^p \mathrm{d} m (x) \Big)^{1/p} + \Big( \int_{X_\L} |\a_\L (\g_2, \g_1^{-1} \cdot x)|_{S_\L}^p \, \mathrm{d} m (x) \Big)^{1/p} \\
         & = \Big( \int_{X_\L} |\a_\L(\g_1, x) |_{S_\L}^p \mathrm{d} m (x) \Big)^{1/p} + \Big( \int_{X_\L} |\a_\L (\g_2, x)|_{S_\L}^p \, \mathrm{d} m (x) \Big)^{1/p}. 
    \end{align*}
    Using this on some $\g \in \G$ to reason by induction on the length $| \g|_{S_\G}$ gives the desired inequality.
\end{proof}

\paragraph{Mutually cobounded measure equivalence and orbit equivalence}


Aside from imposing integrability conditions on some ME-coupling $(\Omega, X_\L, X_\G, m) $ between two countable groups $\L$ and $\G$, our results will require a hypothesis of different nature, this time regarding the relative position of the two fundamental domains $X_\L$ and $X_\G$ inside the coupling space $\Omega$.

\begin{defn}
    Given two countable groups $\L$ and $\G$, we say that a ME-coupling $(\Omega, X_\L, X_\G, m) $ between $\L$ and $\G$ is \textit{cobounded from $\G$ to $\L$} if there exists a finite set $F_\L \subset \L$ such that: 
    \begin{equation*}
         X_\G \subset \bigsqcup_{\l \in F_\L} \l X_\L.
    \end{equation*}
    
    We say that the ME-coupling $(\Omega, X_\L, X_\G, m) $ is \textit{mutually cobounded} if it is cobounded from $\L$ to $\G$ and from $\G$ to $\L$, that is, there exist finite sets $F_\L \subset \L$ and $F_\G \subset \G$ such that: 
    \begin{equation*}
        X_\G \subset \bigsqcup_{\l \in F_\L} \l X_\L \quad \text{ and } \quad  X_\L \subset \bigsqcup_{\g \in F_\G} \g X_\G.
    \end{equation*} 
\end{defn}

Our results will often require mutually cobounded $\textrm L^p$-ME. The following lemma states that if a ME-coupling is cobounded, up to changing one of the two groups by a finite product, we can assume that one of the two fundamental domains is contained in the other, and this without losing any integrability. Similar statements can be found in \cite[2.1.7]{shalom2004harmonic} or in the proof of \cite[Claim 4.5]{DKLMT-hyperbolicity}.

\begin{lem}\label{Putting a fundamental domain inside the other}
    Let $\L$ and $\G$ be two countable groups and $(\Omega, X_\L, X_\G, m) $ a ME-coupling,  that is $\Lrm^p$-integrable for some $1 \leq p \leq \infty$ and cobounded from $\G$ to $\L$. Then there exists a finite group $K$ and a ME-coupling $(\overline{\Omega}, Y_\L, Y_{\G \times K}, \overline{m} ) $ between $\L$ and $\G \times K$ such that $Y_{\G \times K} \subseteq  Y_\L$ and is $\Lrm^p$-integrable.
\end{lem}

\begin{proof}
Let $F_\L$ be such that $ X_\G \subset \bigsqcup_{\l \in F_\L} \l X_\L$ and $K$ be any finite group admitting a simply transitive action on $F_\L$. Fix $f_0 \in F_\L$.
Consider: \\ 
$\bullet$ $ \overline{\Omega} = \Omega \times F_\L$, where the $\L$-action on $\overline{\Omega}$ is $\l \star(\omega, f) := (\l.\omega, f)$ for $(\omega, f) \in \overline{\Omega}$ and $\l \in \L$, and the $(\G \times K)$-action on $\overline{\Omega}$ is given by  $(\g, k)\star(\omega, f) = (\g.\omega, k.f)$ for $(\omega, f) \in \overline{\Omega}$ and $(\g, k) \in \G \times K$. \\
$\bullet$ $Y_\L =  \bigsqcup_{f \in F_\L} f.X_\L \times \{ f \}$, \\
$\bullet$ $Y_{\G \times K} = \bigsqcup_{f \in F_\L} (X_\G\cap f .X_\L) \times \{ f \}$, \\
$\bullet $ $\overline{m} = m \otimes c$ where $c$ is the counting measure on $F_\L$. \\
It is easy to check that the quadruplet $(\overline{\Omega}, Y_\L, Y_{\G \times K}, \overline{m} ) $ is a ME-coupling. 
It is also clear that $Y_{\G \times K} \subseteq  Y_\L$. 

Denote by $c_\L$ and $c_\G$ the cocycles associated to the ME-coupling $(\Omega, X_\L, X_\G, m) $. The cocycles $\overline{c}_\L : (\G \times K) \times Y_\L \to \L$ and $\overline{c}_{\G \times K} : \L \times Y_{\G \times K} \to \G \times K$ associated to the new ME-coupling $(\overline{\Omega}, Y_\L, Y_{\G \times K}, \overline{m} ) $ are defined by: \\
$\bullet$ For $y = ( f.x, f) \in Y_\L$ and $(\g, k) \in \G \times K$, we have $\overline{c}_\L((\g, k), y) = f c_\L (\g, x) f^{-1}$. \\
$\bullet$ For $y = (f.x , f) \in Y_{\G \times K}$, where $f.x \in X_\G$, $x\in X_\L $ and for $\l \in \L$, we have: $\overline{c}_{\G \times K} (\l , y) = (c_\G(\l f, x), k)$, where $k$ is the unique $k \in K$ so that $\g. \l f. x \in k f. X_\L$.
It is now easy to see from these formulas that the integrability of the cocycles $\overline{c}_\L$ and $\overline{c}_{\G \times K}$ will be the same as for the original cocycles $c_\L$ and $c_\G$.
\end{proof}

We close this section by discussing the related notion of orbit equivalence (OE). We will not really use the definition of OE, but we recall that two countable groups $\L$ and $\G$ are OE if and only if there exists a ME-coupling $(\Omega, X_\L, X_\G, m) $ between them satisfying $X_\L = X_\G$.
In particular, such a coupling is mutually cobounded. Notice that while Lemma \ref{Putting a fundamental domain inside the other} allows us to put one fundamental domain inside the other in the situation of a mutually cobounded ME-coupling, we cannot always get a coupling satisfying $X_\L = X_\G$. Hence $\mathrm L^p$-OE (that is, the existence of a ME-coupling $(\Omega, X_\L, X_\G, m) $ satisfying $X_\L = X_\G$ that is $\mathrm L^p$-integrable) is a relation stronger than mutually cobounded $\mathrm L^p$-ME.


\subsection{Finiteness properties and filling functions}

In this section we introduce a combinatorial version of filling functions, these can be seen as a possible way of extending the more classical Dehn function to higher dimensions. Before this, we recall the definition of the finiteness property $F_d$, which will be a standing assumption for most of our results. Most of these definitions and properties are presented in \cite[Section 6.3]{bader-sauer} and \cite[Section 2]{young-filling-nilpotent}.

\begin{defn}
    Let $\G$ be a discrete group and $d \in \na_{>0}$. We say that $\G$ is \textit{of type $F_{d}$} if there exists a contractible simplicial complex $X$ on which $\G$ acts properly with finitely many orbits on its $d$-skeleton.
\end{defn}

With this definition at hand, we now define combinatorial filling functions, first for simplicial complexes.
Let $X$ be a simplicial complex, we denote the integer valued simplicial chain complex of $X$ by $C_*(X, \ze)$, with boundary operator $\partial$ and $Z^k(X, \ze) =C^k(X, \ze) \cap \ker \partial$. Recall that for $c = \sum_{\s \in X^{(k)}} a_\s \s$, the $\ell^1$-norm of $c$ is $||c||_1 = \sum_{\s \in X^{(k)}} |a_\s|$.

The \textit{(combinatorial) filling volume} of an integral cycle $c \in Z_{k-1}(X, \ze)$ is defined as:
\begin{equation*}
    \mathrm{cFV}^k(c):= \min_{b \in C_k(X, \ze), \partial b = c} ||b||_1.
\end{equation*}

\begin{defn}
    Let $k\in \na_{\geq2}$ and suppose that the complex $X$ is $(k-1)$-connected. We define the \textit{$k$-th combinatorial filling function} $\mathrm{cFV}^k$ as:
    \begin{equation*}
        \mathrm{cFV}_X^k(t) := \sup_{c \in Z_{k-1}(X, \ze),  ||c||_1 \leq t} \mathrm{cFV}^k(c). 
    \end{equation*}
\end{defn}

We may now extend this definition to groups. For this, we need to introduce the relation on functions under which this will be well-defined.
For two increasing functions $f, g : \na \to \re_{>0}$, the relation $f \lesssim g$ means that there exists $C>0$ such that for every $n > 0 $ we have $f(n) \leq C g(Cn + C) + Cn + C$, and $f \sim g$ means that $f \lesssim g$ and $g \lesssim f$. 

If $\G$ is a discrete group acting by isometries properly and cocompactly on two $(k-1)$-connected simplicial complexes $X$ and $Y$, we have
\begin{equation*}
    \mathrm{cFV}_X^k \sim \mathrm{cFV}^k_Y
\end{equation*}
(see \cite[Corollary 3]{alonso-wang-pride} for higher Dehn functions, the argument also applies to the combinatorial filling functions $\mathrm{cFV}^j$). 
Hence for a group $\G$ of type $F_{k+1}$, we may define (up to the relation $\sim$) the combinatorial filling function $\mathrm{cFV}_\G^k$ of $\G$ as $\mathrm{cFV}_X^k$ for any $(k-1)$-connected simplicial complex $X$ with such an action.

We say that a function $f: \na \to \re_{>0}$ has \textit{polynomial growth} or is \textit{polynomially bounded} if there exists $d \in \re_+$ such that $f \lesssim t^d$.
If there exists $d \in \re_{>0}$ such that $f \sim t^d$ we will say that $f$ is \textit{polynomial}.
Wenger exhibited examples of (nilpotent) groups whose Dehn function $\d_\G$ is polynomially bounded but not polynomial \cite{wenger-non-polyn-dehn}.
In any case, our results will barely be sensitive to the distinction between polynomial and polynomially bounded.
If $\mathrm{cFV}_\G^k$ is polynomially bounded, we denote 
\begin{equation*}
    \deg \mathrm{cFV}^k_\G := \inf \{ d \in \re_{>0},  \mathrm{cFV}_\G^k \lesssim t^d \}.
\end{equation*}

The filling functions $\mathrm{cFV}_\G^{k+1}$ are related to their homotopical variants, the higher Dehn functions $\d_\G^k$ \cite[Section 2]{young-filling-nilpotent}. Indeed, for $k \geq 3$, we have $\mathrm{cFV}_\G^{k+1} \sim \d_\G^k$. For $k = 1, 2$, the relationship is more subtle. For $k = 2$, it is known that $\d^2_\G \lesssim  \mathrm{cFV}_\G^3$ and that there exist groups for which this inequality is strict \cite{abrams-brady-dani-young}. For $k = 1$, it is conjectured that the opposite behaviour holds, that is $\mathrm{cFV}_\G^2 \lesssim\d_\G := \d_\G^1$ \cite[2.29]{brady-kropholler-soroko}. This is nearly known to be true, we record the following result, which is a particular case of \cite[2.28]{brady-kropholler-soroko} applied to the convex function $t \mapsto t^d$.

\begin{lem}\label{Dehn function is nearly smaller than Filling Area}
    Let $\G$ be a finitely presented group. If for some $d \in \re_{>1}$ we have that $\d_\G \lesssim t^d$, then $\mathrm{cFV}_\G^{2} \lesssim t^d$. In particular $\deg \mathrm{cFV}^2_\G \leq d$.
\end{lem}

One can expect $\deg \mathrm{cFV}^k_\G $ to attain any real value. Indeed, there are examples of groups $\G$ having higher Dehn functions $\d_\G^k \sim t^d$ where $d$ ranges over a dense subset of $ [ (k+1)/k,\infty)$ \cite[Corollary E]{brady-bridson-forester-shankar}.

\section{A quantitative version of polynomial cohomology}\label{Section: A quantitative version of polynomial cohomology}

The goal of this section is to prove Theorem \ref{Intro theorem: Quantitative isomorphism of polynomial cohom onto ordinary cohom}, which can be thought of as a quantitative analogue of \cite[Proposition 6.14]{bader-sauer}. We restate the result here.

\begin{thm} \label{Quantitative isomorphism of polynomial cohom onto ordinary cohom} 
Let $X$ be a contractible free simplicial $\G$-complex with cocompact $(d+1)$-skeleton. Suppose that all the filling functions $ \mathrm{cFV}_X^i$ are polynomially bounded for $1 \leq i \leq d+1$. Let $V$ be an isometric representation of $\G$ on some Banach space $V$. For $1 \leq k \leq d$, there exist a constant $N_k>0$ such that for all $N > N_k$ and $M \geq N$ the comparison map:
 \begin{equation*}
     H_{\mathrm{pol}: M \to N}^k(\G, V) \to H^k(\G, V)
 \end{equation*}
is surjective (the same holds for their respective reduced versions). Moreover, there exists a function $\a_k: \re_{>0} \to \re$ such that for all $N > N_k$ and $M > \a_{k-1}(N)$ the comparison map:
 \begin{equation*}
     H_{\mathrm{pol}: M \to N}^k(\G, V) \to H^k(\G, V)
 \end{equation*}
is an isomorphism (the same holds for their respective reduced versions). We can choose $N_1 = 1$ and for $k \geq 2$, $N_k = \prod_{j= 2}^k \mathrm{deg}( \mathrm{cFV}_X^j)$ and the functions $\a_i$ are defined inductively by the formula $\a_1(N) = N+1$ and for $i \geq 2$, we have:
    \begin{equation*}
        \a_i(N) = \max \{ (N+1) \prod_{j= 2}^i  \mathrm{deg}( \mathrm{cFV}_X^j),  \prod_{j= 2}^i  \mathrm{deg}( \mathrm{cFV}_X^j) + \a_{i-1}(N) \}.
    \end{equation*}
If the filling functions $ \mathrm{cFV}_X^j$ are polynomial for $2 \leq j \leq k$, then surjectivity of the comparison map holds for $N = N_k$ and isomorphism holds for $N = N_k$ and $M = \a_{k-1}(N) = \a_{k-1}(N_k)$. 
\end{thm}

In this section we will first introduce usual group cohomology, a quantitative variant of polynomial cohomology and the dual theory of polynomial decay subalgebras of $\ell^1$. The proof of Theorem \ref{Quantitative isomorphism of polynomial cohom onto ordinary cohom} then consists on revisiting the isomorphism between group cohomology and equivariant simplicial cohomology in a quantitative way.

\subsection{Quantitative polynomial cohomology}

Polynomial cohomology is defined by restricting our attention to cochains with a polynomial growth condition. 
Following \cite{bader-sauer}, we introduce a quantitative version of polynomial cohomology in order to take into account the degree of the polynomial in question.

We first recall how to define (continuous) group cohomology as in \cite[Chapter  IX]{borel-wallach}. Even though we will restrict our attention to discrete groups, we will give these definitions in the locally compact setting.

Let $G$ be a locally compact second countable group and $(\rho, V)$ be a continuous isometric representation of $G$ on a Banach space. Let $k \in \na$. We define the space $C^k (G , V)$ of continuous maps $c: G^{k+1} \to V$ and we view it as a Fréchet space with the topology of uniform convergence on compact subsets. Consider the closed subspace of \textit{continuous $G$-equivariant cochains} $C^k (G , V)^G$ consisting of continuous cochains $c: G^{k+1} \to V$, such that for all $g , g_0 , \ldots, g_k \in G$ we have:
\begin{equation*}
    c(gg_0, \ldots, gg_k) = \rho(g) c(g_0, \ldots , g_k).
\end{equation*}

We have the usual coboundary operator $d : C^k(G, V) \to C^{k+1}(G, V)$ defined by the formula:
\begin{equation*}
    (dc)(g_0, \ldots, g_{k+1}) = \sum_{i = 0}^{k+1} (-1)^{i} c(g_0, \ldots ,g_{i-1}, g_{i+1} ,\ldots, g_{k+1}),
\end{equation*}
for every $c \in C^k(G, V)$ and $g_0, \ldots, g_{k+1} \in G$. This operator restricts to continuous $G$-equivariant cochains and gives a bounded operator $d:  C^k(G, V)^G \to C^{k+1}(G, V)^G$. We denote by 
\begin{equation*}
    Z^{k}(G, V) := \ker ( d |_{C^{k}(G, V)^G}), \quad
    B^{k}(G, V) := d ( C^{k-1}(G, V)^G).
\end{equation*}

\begin{defn}
    For $k \in \na$, we define the $k$-th \textit{continuous cohomology group with coefficients in the representation} $(\rho, V)$ (resp. \textit{reduced continuous cohomology group with coefficients in} $(\rho, V)$) to be the quotient:
    \begin{equation*}
        H_{\mathrm{ct}}^k(G, V) : = Z^{k}(G, V) /  B^{k}(G, V), \quad  \text{(resp. } \overline{H}_{\mathrm{ct}}^k(G, V) : = Z^{k}(G, V) /  \overline{B^k(G, V)} ).
    \end{equation*}
\end{defn}

Now let $G$ be a compactly generated locally compact group, with compact generating set $S$ and associated word length $| \cdot |$. Fix $N \in \re_{\geq 0 }$. We define the space $C_{\mathrm{pol}\leq N}^k(G, V)$ of \textit{continuous cochains of polynomial growth at most $N$} as the space consisting of $c \in C^k(G,V)$ for which there exists a constant $M$ such that for every $g_0 , \ldots, g_k \in G$ we have:
\begin{equation*}
    ||c(g_0, \ldots, g_k)|| \leq M (1 + |g_0| + \ldots +|g_k|)^N.
\end{equation*}
We endow $C_{\mathrm{pol}\leq N}^k(G, V)$ with the topology of uniform convergence on compact subsets. Notice that endowed with this topology, $C_{\mathrm{pol}\leq N}^k(G, V)$ is not closed inside $C^k(G, V)$, hence not complete.

The operator $d$ restricts to cochains of polynomial growth at most $N$ and to $G$-equivariant cochains of polynomial growth at most $N$. Hence we obtain two maps $d : C_{\mathrm{pol}\leq N}^k(G, V) \to C_{\mathrm{pol}\leq N}^{k+1}(G, V)$ and $d : C_{\mathrm{pol}\leq N}^k(G, V)^G \to C_{\mathrm{pol}\leq N}^{k+1}(G, V)^G$. We denote by 
\begin{equation*}
    Z_{\mathrm{pol}\leq N}^{k}(G, V) := \ker ( d |_{C_{\mathrm{pol}\leq N}^{k}(G, V)^G}).
\end{equation*}
and for $M \geq N$
\begin{equation*}
    B_{\mathrm{pol}: M \to N}^{k}(G, V) := d ( C_{\mathrm{pol}\leq M}^{k-1}(G, V)^G) \cap  C_{\mathrm{pol}\leq N}^{k}(G, V)^G.
\end{equation*}

\begin{defn}
    For $k \in \na$ and $M\geq N \geq 0$, we define the \textit{$k$-th quantitative polynomial cohomology space of degrees $M$ and $N$} by:
    \begin{equation*}
        H_{\mathrm{pol}: M \to  N}^k(G, V) : = Z_{\mathrm{pol}\leq N}^{k}(G, V) / B_{\mathrm{pol}: M \to N}^{k}(G, V).
        \end{equation*}
        The definition of its corresponding reduced version is a little bit more delicate. The space $Z_{\mathrm{pol}\leq N}^{k}(G, V)$ carries the subspace topology of $Z^{k}(G, V)$ (and it may not be a closed subspace). We write $\overline{B_{\mathrm{pol}: M \to N}^k(G, V)}$ for the closure of $B_{\mathrm{pol}: M \to N}^k(G, V)$ inside $Z_{\mathrm{pol}\leq N}^{k}(G, V)$. We may define the \textit{$k$-th reduced quantitative polynomial cohomology space of degrees $M$ and $N$} by:
        \begin{equation*}
        \overline{H}_{\mathrm{pol}: M \to  N}^k(G, V) : = Z_{\mathrm{pol}\leq N}^{k}(G, V) /  \overline{B_{\mathrm{pol}: M \to N}^k(G, V)}.
    \end{equation*}
\end{defn}

These two spaces inherit a natural quotient topology, and $\overline{H}_{\mathrm{pol}: M \to  N}^k(G, V)$ is the largest Hausdorff quotient of $H_{\mathrm{pol}: M \to  N}^k(G, V)$.

The inclusion of polynomial cochains $C_{\mathrm{pol}\leq N}^{k}(G, V) \to C^{k}(G, V)$ induces natural continuous maps \begin{align*}
    H_{\mathrm{pol}:M \to N}^k(G, V) &\to H_\mathrm{ct}^k(G, V), \\
     \overline{H}_{\mathrm{pol}:M \to N}^k(G, V) &\to \overline{H}_\mathrm{ct}^k(G, V).
\end{align*}
that we call \textit{comparison maps}. The goal of this section is to prove Theorem \ref{Intro theorem: Quantitative isomorphism of polynomial cohom onto ordinary cohom}, which describes situations in which these comparison maps are isomorphisms for well chosen $N$ and $M$.

\begin{rem} To simplify this discussion, let $\G$ be a discrete group.
    Notice that when endowing the spaces $C^k(\G, V)$ with the $\s$-finite measure 
    \begin{equation*}
        \mu_k (\g_0, \ldots, \g_k)= (1 + |\g_0| + \ldots +|\g_k|)^{-N},
    \end{equation*}
    we may see the space $C^k_{\mathrm{pol}\leq N}(\G, V)$ as the weighted $\mathrm{L}^\infty$-space $\mathrm{L}^\infty (\G^{k+1}, V, \mu_k)$. The space $\mathrm{L}^\infty (\G^{k+1}, V, \mu_k)$ can be naturally endowed with the topology of the norm 
    \begin{equation*}
        ||c||_{\infty, \mu_k} : = \sup \{ ||c(\xi)|| \mu_k(\xi), \xi \in \G^{k+1} \} .
    \end{equation*}
    Nevertheless, this is not the topology we are endowing $C^k_{\mathrm{pol}\leq N}(\G, V)$ with.
\end{rem}

\subsection{Polynomially decaying algebras}

We introduce some preliminary tools that will be necessary for the proof of Theorem \ref{Intro theorem: Quantitative isomorphism of polynomial cohom onto ordinary cohom}. These tools are analogous to those in \cite[6.2]{bader-sauer}.

Let $\G$ be a discrete finitely generated group with finite generating set $S$ and associated word length $| \cdot |$. 

 For a function $f : \G \to \co$ and $N \in \re_{\geq 0}$ we let 
\begin{equation*}
    p_N(f) = \sum_{\g \in \G} |f(\g)| (1 + |\g|)^N. 
\end{equation*}
Define $\mathcal{S}_N(\G) = \{f : \G \to \co, \, p_N(f) < \infty \}$. This is an algebra for convolution, as subadditivity of the word length $| \cdot |$ implies that for any $f, f' : \G \to \co$ we have:
\begin{equation*}
    p_N(f * f') \leq  p_N(f) p_{N}(f').
\end{equation*}
The space $\mathcal{S}_N(\G)$ endowed with the norm $p_N$ can be seen as a weighted $\ell^1$-space for the measure $m$ on $\G$ given by $m(\g) = (1 + |\g|)^N$. Hence the space $\mathcal{S}_N(\G)$ is complete, and therefore a Banach algebra. In fact every $\mathcal{S}_N(\G)$ is an isometric copy of the usual $\ell^1$-space.

Let $X$ be a free simplicial $\G$-complex with cocompact $d$-skeleton for some $d \in \na$. The simplicial chain complex of $X$ is denoted by $C_*(X, \co)$.

Endow each $d$-simplex in $X$ with the Euclidean metric coming from the standard $d$-dimensional simplex in $\re^{d+1}$ and endow $X$ with the corresponding path metric. Choose a base vertex $v_0$ in $X$. We define a weight function $w$ on an ordered simplex $\s$ by setting $w(\s)$ to be the distance in the path metric of the first vertex of $\s$ to $v_0$. We define $\mathcal{S}_N C_* (X)$ to be the completion of $C_*(X, \co)$ for the norm $p_N^X$, defined on a simplicial $d$-chain $\tau = \sum a_\s \s$ by:
\begin{equation*}
    p_N^X(\tau) = \sum_{\s} |a_\s| (1 + w(\s))^N.
\end{equation*}

The following lemma explains the relation between polynomial cohomology and the Banach algebra $\mathcal{S}_N(\G)$.

\begin{lem} \cite[ Lemma 6.3]{bader-sauer} \label{Quantitative extension versus growth}
    Let $V$ be an isometric representation on a Banach space. For every $N > 0$, the inclusion
    \begin{equation*}
        \co [\G^{n+1}] = \co[\G] \otimes \ldots \otimes \co[\G] \subset \mathcal{S}_N(\G) \hat{\otimes} \ldots \hat{\otimes} \mathcal{S}_N(\G)
    \end{equation*}
    induces isomorphisms:
    \begin{equation*}
        C_{\mathrm{pol}\leq N} (\G^{n+1}, V)^{\G} \simeq \mathrm{hom}_{\mathcal{S}_N(\G)} (\mathcal{S}_N(\G) \hat{\otimes} \ldots \hat{\otimes} \mathcal{S}_N(\G), V).
    \end{equation*}
    where $\hat{\otimes}$ denotes the projective tensor product.
\end{lem}

\begin{proof}
    The norm on the projective tensor product $\mathcal{S}_N(\G) \hat{\otimes} \ldots \hat{\otimes} \mathcal{S}_N(\G)$ is the completion of the norm on $\co [\G^{n+1}]$ given by
    \begin{equation*}
        q_N\left(\sum_{\overline{\g} \in \G^{n+1}} a_{\overline{\g}}(\g_0, \ldots, \g_n)\right) = \sum_{\overline{\g} \in \G^{n+1}}  |a_{\overline{\g}}|  (1 + |\g_0| + \ldots + |\g_n|)^N.
    \end{equation*}
    This means that a $\G$-equivariant map $c : \G^{n+1} \to V$, seen as a $\co \G$-equivariant linear map $c: \co [\G^{n+1}] \to V$, extends continuously as a continuous linear map $c: \mathcal{S}_N(\G)^{\otimes (n+1)} \to V$ exactly when $c \in C_{\mathrm{pol}\leq N} (\G^{n+1}, V)^{\G}$. The $\mathcal{S}_N(\G)$-equivariance of the extension follows from density of $\co \G$ inside $\mathcal{S}_N(\G)$.
\end{proof}

\subsection{Use of combinatorial filling functions}

It is known that if a group $\G$ admits a proper and cocompact action on a contractible simplicial complex $X$, then group cohomology of $\G$ is isomorphic to a simplicial $\G$-equivariant version of cohomology on the complex $X$. The following proposition shows that by assuming that homological filling functions of $X$ are polynomially bounded, we obtain that the usual chain map that induces this isomorphism satisfies explicit estimates for the $\ell^1$-norm and the weighted norms $p_N$.

\begin{prop}\label{L^1-norm of g_k}
    Let $X$ be a contractible free simplicial $\G$-complex with cocompact $(d+1)$-skeleton. Suppose that all the filling functions $ \mathrm{cFV}_X^k$ are polynomially bounded for $2 \leq k \leq d+1$. Then there exists a chain map $g_* : \co [\G ^{*+1}] \to C_*(X, \co)$ such that for every $1 \leq k \leq d$ there exists $N_k \in \re_{>0}$ such that for every $N> N_k$ there exists $C = C(N, k)>0$ such that we have:
    \begin{equation*}
        ||g_k(\g_0, \ldots, \g_k )||_1 \leq C p_N(\g_0, \ldots, \g_k),
    \end{equation*}
    for every $(\g_0, \ldots, \g_k ) \in \G ^{k+1}$. The numbers $N_k$ are defined by: $N_1 = 1$ and for $k \geq 2$, we have $N_k = \mathrm{deg}( \mathrm{cFV}_X^k) N_{k-1}$ so that:
        \begin{equation*}
        N_k = \prod_{j= 2}^k \mathrm{deg}( \mathrm{cFV}_X^j).
    \end{equation*}
    If the filling functions $ \mathrm{cFV}_X^j$ are polynomial for $2 \leq j \leq k$, then the previous inequality also works for $ N = N_k$.
\end{prop}

\begin{proof}
    We first construct an integral chain contraction $h_* : C_*(X, \co) \to C_{*+1}(X, \co)$ similar to the one in the proof of \cite[Proposition 6.13]{bader-sauer} (but we use the more classical filling functions $ \mathrm{cFV}_X^j$ instead of their weighted versions). First define $h_{-1}: \co \to C_0(X, \co)$ to be the map sending $1$ to some base vertex $v_0 \in X^{(0)}$. For some $v \in X^{(0)}$ define $h_0(v)$ to be a simplicial geodesic from $v$ to $v_0$, this map naturally satisfies $\partial h_0 + h_{-1} \partial = \Id$ and is integral. 
    
    We now construct inductively the maps $h_k : C_*(X, \co) \to C_{*+1}(X, \co)$. For $k \geq 1$ and some $k$-simplex $\s$, we define $h_k(\s)$ to be a primitive of the boundary $\s - h_{k-1} (\partial \s)$, that is, $ \partial h_k(\s) =  \s - h_{k-1} (\partial \s)$ and we choose it to be optimal for the $\ell^1$-norm, that is, $||h_k(\s)||_1 = \mathrm{cFV}^{k+1}(\s - h_{k-1} (\partial \s))$.

    We now turn to define the chain map $g_* :\co [\G ^{*+1}] \to C_*(X, \co)$ using the chain contraction $h_*$. First define $g_0(1) = v_0$ to be the same base vertex $v_0 \in X^{(0)}$ that was chosen for $h_{-1}$ and extend by $\co \G$-equivariance. For $k \geq 1$, suppose that the maps $g_0, \ldots, g_{k-1}$ have been defined. For $(\g_1, \ldots, \g_k) \in \G^k$, we define:
    \begin{equation*}
        g_k(1, \g_1, \ldots, \g_k) = h_{k-1} \circ g_{k-1} \circ \partial (1, \g_1, \ldots, \g_k)
    \end{equation*}
    and we extend it to all of $\co \G^{\otimes^{k+1}}$ by $\co \G$-equivariance.

    Let's now focus on showing the inequality. For $k = 1$, the map $g_1$ sends the couple $(1, \g)$ to a simplicial geodesic joining $v_0$ to $\g. v_0$. This means that
    \begin{equation*}
        ||g_1(1, \g)||_1 \leq C (1 + |\g|),
    \end{equation*}
    which then implies the desired inequality on all $\co \G^{\otimes^{2}}$ by $\co \G$-equivariance.
    For $k \geq 2$, suppose that $g_i$ satisfies the required inequality up $i \leq k-1$. Notice that for $(\g_1, \ldots, \g_k) \in \G^{k}$, the chain $g_{k-1} \circ \partial (1, \g_1, \ldots, \g_k)$ is a boundary. Hence we have for any $N > \deg \mathrm{cFV_X^k}$ and $N' > N_{k-1}$:
    \begin{align*}
        ||g_k(1, \g_1, \ldots, \g_k)||_1 
        & = || h_{k-1} \circ g_{k-1} \circ \partial (1, \g_1, \ldots, \g_k) ||_1 \\
        & \leq C_1  ||g_{k-1} \circ \partial (1, \g_1, \ldots, \g_k) ||_1^N \\
        & \leq C_2  (||\partial (1, \g_1, \ldots, \g_k) ||_1^{N'})^{N} \\
        & \leq C_3  (1 + |\g_1| + \ldots + |\g_k|)^{N'N},
    \end{align*}
    for all $(\g_1, \ldots, \g_k) \in \G^{k}$. Thus we obtain the inequality in the statement by defining $N_k : = N_{k-1} \deg \mathrm{cFV_X^k}$.
    
\end{proof}

The next proposition is similar in nature to the previous one. Instead of bounding the $\ell^1$-norm, we now bound the weighted norms $p_N$ for the same chain map $g_*$. In fact, the next proposition contains the previous one as a particular case, but yields slightly worse estimates than those from the previous proof.

\begin{prop}\label{weighted norms of g_k}
    Let $X$ be a contractible free simplicial $\G$-complex with cocompact $(d+1)$-skeleton. Suppose that all the filling functions $ \mathrm{cFV}_X^j$ are polynomially bounded for $2 \leq j \leq d+1$. Then the chain map $g_* : \co [\G ^{*+1}] \to C_*(X, \co)$ from Proposition \ref{L^1-norm of g_k} satisfies that for all $N>0$, there exists a constant $\a_k(N) \in  \re_{>0}$ such that for every $M > \a_k(N)$ there exists $C = C(N, k)>0$ such that we have:
    \begin{equation*}
        p_N (g_k(\g_0, \ldots, \g_k )) \leq C p_{M}(\g_0, \ldots, \g_k), 
    \end{equation*}
    for all $(\g_0, \ldots, \g_k ) \in \G ^{k+1}$. The constants $\a_k(N)$ are defined inductively as follows: $\a_1(N) = N+1$ and for $k \geq 2$, we have:
    \begin{equation*}
        \a_k(N) = \max \{ (N+1) \prod_{j= 2}^k  \mathrm{deg}( \mathrm{cFV}_X^j),  \prod_{j= 2}^k  \mathrm{deg}( \mathrm{cFV}_X^j) + \a_{k-1}(N) \}.
    \end{equation*}
    If the filling functions $ \mathrm{cFV}_X^j$ are polynomial for $2 \leq j \leq k$, then the previous inequality also works for $M = \a_k(N)$.
\end{prop}

\begin{proof}
    We combine ideas of the proof of \cite[Lemma 6.7]{bader-sauer} with Proposition \ref{L^1-norm of g_k}. Fix $N>0$. We first deal with $k = 1$. For $\g \in \G$, the chain $g_1(1, \g)$ is just a simplicial geodesic from $\g. v_0 $ to $v_0$. Call the vertices appearing in the geodesic $v_i \in V^{(0)}$, where $0 \leq i \leq n$. We have $w(v_i) \leq l(\g)$ and $n \leq C l(\g))$ so: 
    \begin{equation*}
        p_N^X(g_1(1, \g)) = \sum_{i= 0}^n (1 + w(v_i))^N \leq C (1 + |\g|)^{N+1}.
    \end{equation*}
    This means that we can choose $\a_1(N) = N+1$.
    
    Now let $k \geq 2$ and suppose that the inequality has been proven for $i \leq k-1$. Let $(\g_1, \ldots, \g_k ) \in \G^k$, and denote by $b := g_k (1, \g_1, \ldots, \g_k )$ and by $c : = g_{k-1} \circ \partial (1, \g_1, \ldots, \g_k )$. Since $g_*$ is a chain map, we have that $\partial b = c$ and by definition of $h_{k-1}$ we also have that $||b||_1 = \mathrm{cFV}^{k}(c)$. We now mimic the estimates in the proof of \cite[Lemma 6.7]{bader-sauer} to obtain:

    \begin{equation*}
        p_N^X(b) \leq ||b||_1^{N+1} + ||b||_1 p_N^X(c).
    \end{equation*}
    Proposition \ref{L^1-norm of g_k} says that for any $D > \prod_{j=2}^k \deg \mathrm{cFV_X^j}$ there exists $C>0$ such that
    \begin{equation*}
          ||b||_1 \leq C (1 + |\g_1| + \ldots + |\g_k|)^{D}.
    \end{equation*}
   If we let the numbers $\a_i(N)$ be defined inductively for $i =1, \ldots,  k-1$, our inductive hypothesis says that for every $M > \a_{k-1}(N)$ we have:
   \begin{equation*}
       p_N^X(c) \leq C' p_{M}(1, \g_1, \ldots, \g_k ) =(1 +  |\g_1| + \ldots + |\g_k| )^M.
   \end{equation*}
   Hence we obtain
   \begin{equation*}
        p_N^X(b) \leq ||b||_1^{N+1} + ||b||_1 p_N^X(c) \leq C'' (1 + |\g_1| + \ldots + |\g_k|)^{\max \{ D (N+1), D + M \}}.
   \end{equation*}
   The proposition follows by setting:
   \begin{equation*}
        \a_k(N) := \max \{ (N+1) \prod_{j= 2}^k  \mathrm{deg}( \mathrm{cFV}_X^j),  \prod_{j= 2}^k  \mathrm{deg}( \mathrm{cFV}_X^j) + \a_{k-1}(N) \}.
    \end{equation*}
\end{proof}

\subsection{Comparing quantitative polynomial and ordinary cohomology}

In this section we will prove Theorem \ref{Quantitative isomorphism of polynomial cohom onto ordinary cohom}. We will use the following Lemma, which is the homological interpretation of Propositions \ref{L^1-norm of g_k} and \ref{weighted norms of g_k}.

\begin{lem}\label{Existence of continuous homotopy}
Let $X$ be a contractible free simplicial $\G$-complex with cocompact $(d+1)$-skeleton. Suppose that all the filling functions $ \mathrm{cFV}_X^i$ are polynomially bounded for $1 \leq i \leq d+1$.
    Choose a $\co \G$-equivariant chain map $f_* : C_*(X, \co) \to \co [\G^{*+1}]$, let $N_i, \a_i$ and $g_* : \co [\G ^{*+1}] \to C_*(X, \co)$ be defined as in Propositions \ref{L^1-norm of g_k} and \ref{weighted norms of g_k}. For $k \leq d$, let $N > N_k$ and $M > \a_k(N)$.
    There exist
     $\mathcal{S}_{M}(\G)$-equivariant continuous maps $H_i : \mathcal{S}_{M}(\G)^{\otimes^{i+1}} \to  \mathcal{S}_{N}(\G)^{\otimes^{i+2}}$ for $1\leq i \leq k$ satisfying 
    \begin{equation*}
        f_i g_i - \Id = \partial H_i + H_{i-1} \partial.
    \end{equation*}
    If the filling functions $ \mathrm{cFV}_X^j$ are polynomial for $2 \leq j \leq k$, then the statement holds for $N = N_k$ and $M = \a_k(N) = \a_k(N_k)$.
\end{lem}

\begin{proof}
      Proposition \ref{weighted norms of g_k}, means that for every $i$, the map $g_i$ can be extended continuously to a map $\mathcal{S}_{\a_i(N)}(\G)^{\otimes^{i+1}} \to \mathcal{S}_{N}C_i(X)$ that is $\mathcal{S}_{\a_i(N)}(\G)$-equivariant. Since $M >  \a_k (N) \geq \a_i(N)$ for $i \leq k$, we ensure that for every $1 \leq i \leq k$, the map $g_i$ can be extended continuously to a map  $\mathcal{S}_{M}(\G)^{\otimes^{i+1}} \to \mathcal{S}_{N}C_i(X) $ that is $\mathcal{S}_{M}(\G)$-equivariant. Since the $\co \G$-equivariant maps $f_i$ are determined by a finite number of values, the composition $f_i g_i : \co [\G^{i+1}] \to \co [\G^{i+1}]$ can be $ \mathcal{S}_{M}(\G)$-equivariantly extended to a map $f_i g_i : \mathcal{S}_{M}(\G)^{\otimes^{k+1}} \to \mathcal{S}_{N}(\G)^{\otimes^{k+1}}$.
    We consider the resolutions
    \begin{align*}
    & \mathcal{S}_{M}(\G)^{\otimes^{k+2}} \xrightarrow{\partial_{k+1}} \ldots \to \mathcal{S}_{M}(\G)^{\otimes^{2}} \xrightarrow{\partial_{1}} \mathcal{S}_{M}(\G) \xrightarrow{\partial_{0}} \co \to 0, \\
        & \mathcal{S}_{N}(\G)^{\otimes^{k+2}}  \xrightarrow{\partial_{k+1}}  \ldots \to \mathcal{S}_{N}(\G)^{\otimes^{2}}  \xrightarrow{\partial_1}  \mathcal{S}_{N}(\G)  \xrightarrow{\partial_0}  \co \to 0
    \end{align*}
    as resolutions by $\mathcal{S}_{M}(\G)$-modules of the trivial module. These two sequences are strongly exact, that is, the maps $\partial_i: \mathcal{S}_{N}(\G)^{\otimes^{i+1}} \to \Im \partial_i \subseteq \mathcal{S}_{N}(\G)^{\otimes^{i}}$ have $\co$-linear continuous sections (in this case, this is just the continuous extension of the map $s_i(\g_1, \ldots, \g_i) = (1_\G,\g_1, \ldots ,\g_i )$). Notice also that the first resolution
    \begin{equation*}
        \mathcal{S}_{M}(\G)^{\otimes^{k+2}} \xrightarrow{\partial_{k+1}} \ldots \to \mathcal{S}_{M}(\G)^{\otimes^{2}} \xrightarrow{\partial_{1}} \mathcal{S}_{M}(\G) \xrightarrow{\partial_{0}} \co \to 0
    \end{equation*}
    consists of free, hence relatively projective $ \mathcal{S}_{M}(\G)$-modules. Since $(f_i g_i - \Id)_{i \leq k}$ gives an $\mathcal{S}_{M}(\G)$-equivariant chain map between our two resolutions, relative projectivity of this resolution yields inductively $\G$-equivariant continuous linear maps $H_i : \mathcal{S}_{M}(\G)^{\otimes^{i+1}} \to  \mathcal{S}_{N}(\G)^{\otimes^{i+2}}$ for $0\leq i \leq k$ (in fact they are $\mathcal{S}_{M}(\G)$-equivariant by linearity and using continuity for passing to the completion) satisfying 
    \begin{equation*}
        f_i g_i - \Id = \partial_{i+1} H_i + H_{i-1} \partial_{i},
    \end{equation*}
    as in the following commutative diagram (where we set $e_i = f_i g_i - \Id$ and drop $\G$ from our notation).

\begin{tikzcd}[scale=0.7em]
	{\mathcal{S}_{M}^{\otimes^{i+2}}} & {\mathcal{S}_{M}^{\otimes^{i+1}}} &  \ldots & {\mathcal{S}_{M}^{\otimes^{2}}} & {\mathcal{S}_{M}} & {\mathbb{C}} & 0 \\
	{\mathcal{S}_{N}^{\otimes^{i+2}}} & {\mathcal{S}_{N}^{\otimes^{i+1}}} & \ldots & {\mathcal{S}_{N}^{\otimes^{2}}} & {\mathcal{S}_{N}} & {\mathbb{C}} & 0
	\arrow["\partial", from=1-1, to=1-2]
	\arrow["{e_{i+1}}"', from=1-1, to=2-1]
	\arrow["\partial", from=1-2, to=1-3]
	\arrow["{H_i}"{description}, from=1-2, to=2-1]
	\arrow["{e_i}"', from=1-2, to=2-2]
	\arrow["\partial", from=1-3, to=1-4]
	\arrow["\partial", from=1-4, to=1-5]
	\arrow["\partial", from=1-5, to=1-6]
	\arrow["{e_{1}}"', from=1-4, to=2-4]
	\arrow[from=1-6, to=1-7]
	\arrow["{H_0}"{description}, from=1-5, to=2-4]
	\arrow["{e_{0}}"', from=1-5, to=2-5]
	\arrow["0"', from=1-6, to=2-6]
	\arrow["\partial"', from=2-1, to=2-2]
	\arrow["\partial"', from=2-2, to=2-3]
	\arrow["\partial"', from=2-3, to=2-4]
	\arrow["\partial"', from=2-4, to=2-5]
	\arrow["\partial"', from=2-5, to=2-6]
	\arrow[from=2-6, to=2-7]
\end{tikzcd}

\end{proof}

We can now prove our quantitative analogue of \cite[Proposition 6.14]{bader-sauer}.

\begin{proof}[Proof of Theorem \ref{Quantitative isomorphism of polynomial cohom onto ordinary cohom}]
    Fix $N > N_k$ and $M > \a_{k-1} (N)$. Choose a chain map $f_* : C_*(X, \co) \to \co [\G^{*+1}]$. We will first show that the comparison map is surjective. Let $c \in Z^k(\G, V)$. We want to show that there exists $c' \in Z_{\mathrm{pol} \leq N}^k(\G, V)$ and some $b \in B^k(\G, V)$ such that $c = c' + b$. The chain maps $f_*$ and $g_*$ are chain homotopy equivalences of $\co \G$-modules, hence there exist maps $H_* : \co [\G^*] \to \co [\G^{*+1}]$ such that 
    \begin{equation*}
        f_k g_k - \Id = \partial H_k + H_{k-1} \partial 
    \end{equation*}
    hence for $\xi \in \co[\G^{k+1}] $ we have:
    \begin{equation*}
        c( f_k g_k(\xi)) - c(\xi) = (d  H_{k-1}^* c)(\xi).
    \end{equation*}
    Proposition \ref{L^1-norm of g_k} shows that the map $g_k$ can be extended continuously to a map from $\mathcal{S}_{N}(\G)^{\otimes^{k+1}}$ to $\ell^1 C_k(X)$. Moreover the map $c \circ f_k$ can be extended to a map from $\ell^1 C_k(X)$ to $\ell^1(\G)$ as it is determined by a finite number of values. This means that $c' := c \circ f_k \circ g_k$ can be extended to all of  $\mathcal{S}_{N_k}(\G)^{\otimes^{k+1}}$ and hence $c' \in Z^k_{\mathrm{pol} \leq N}(\G, V)$. By setting $b = d  H_{k-1}^* c$, we obtain that $c = c' + b$. This proves surjectivity of the map $H_{\mathrm{pol}: M \to N}^k(\G, V) \to H^k(\G, V)$. By composing with the quotient map $H^k(\G, V) \to \overline{H}^k(\G, V)$, we see from this that the map $ \overline{H}_{\mathrm{pol}: M \to N}^k(\G, V) \to \overline{H}^k(\G, V)$ is also surjective.


    We will now show that the comparison map is injective.
    Let $H_i : \mathcal{S}_{M}(\G)^{\otimes^{i+1}} \to  \mathcal{S}_{N}(\G)^{\otimes^{i+2}}$ be the $\mathcal{S}_{M}(\G)$-equivariant continuous maps obtained in Proposition \ref{Existence of continuous homotopy} for $0 \leq i \leq k-1$.

     Let $b \in  Z_{\mathrm{pol} \leq N}^k(\G, V)$ for which there exists $c \in C^k (\G, V)$ such that $b = dc$. Our goal is to show that there exists $c' \in  C_{\mathrm{pol} \leq M}^{k-1}(\G, V)$ such that $b = dc'$. For $\xi \in \co [\G^{k}]$ we have:
    \begin{equation*}
         c(\xi) =  c( f_k g_k(\xi)) - dc(H_{k-1} \xi)  -  (d  H_{k-2}^* c)(\xi).
    \end{equation*}
    We set $c'(\xi) =  c( f_k g_k(\xi)) - dc(H_{k-1} \xi)$, so that $c = c ' + d  H_{k-2}^* c$. Since $d (d  H_{k-2}^* c) = 0$, we have that $dc' = dc = b$. It remains to show that $c' \in C_{\mathrm{pol} \leq M}^{k-1}(\G, V)$. Set $c_1(\xi) =  c( f_k g_k(\xi))$ and $c_2(\xi) =  dc(H_{k-1} \xi) = b(H_{k-1} \xi)$ so that $c' = c_1 + c_2$. Using the same argument as for the surjectivity part of our proof, Proposition \ref{L^1-norm of g_k} and Lemma \ref{Quantitative extension versus growth} imply that $c_1 \in C^{k-1}_{\mathrm{pol} \leq N}(\G, V) \subset C^{k-1}_{\mathrm{pol} \leq M}(\G, V)$. 
    It remains to show that $c_2 \in  C^{k-1}_{\mathrm{pol} \leq M}(\G, V)$. Lemma \ref{Quantitative extension versus growth} says that $b$ can be extended continuously to a map $\mathcal{S}_{N}(\G)^{\otimes^{k+1}} \to V$ and Lemma \ref{Existence of continuous homotopy} says that $H_{k-1}$ maps $ \mathcal{S}_{M}(\G)^{\otimes^{k}}$ into $\mathcal{S}_{N}(\G)^{\otimes^{k+1}}$. Hence $c_2$ can be extended to $ \mathcal{S}_{M}(\G)^{\otimes^{k}}$ and hence Lemma \ref{Quantitative extension versus growth} says that $c_2 \in C^{k-1}_{\mathrm{pol} \leq M}(\G, V)$. This proves injectivity of $H_{\mathrm{pol}: M \to N}^k(\G, V) \to H^k(\G, V)$.

    The assignment $c \mapsto c' =c_1 + c_2$ is continuous in the topology of convergence on compact subsets as the maps $f_k \circ g_k$ and $H_{k-1}$ are linear and send $\G^{k+1}$ into $\co [\G^{k+1}]$. Hence injectivity also holds for reduced cohomology and the continuous map $c \mapsto c_1$ induces a continuous inverse of the comparison map both in reduced and unreduced cohomology.

\end{proof}

\subsection{Examples and consequences}

We now turn to describing Theorem \ref{Quantitative isomorphism of polynomial cohom onto ordinary cohom} in particular examples.

We state first the case of degree 2 cohomology. With the notation of Theorem \ref{Quantitative isomorphism of polynomial cohom onto ordinary cohom}, we can choose $N_2$ to be $\deg \mathrm{cFV}_\G^2 = \deg \mathrm{ \d_\G}$ by Lemma \ref{Dehn function is nearly smaller than Filling Area}, where $\d_\G$ denotes the Dehn function of $\G$, and we have $\a_1(N) = N+1$.

\begin{cor}\label{Corollary: Comparison theorem in degree 2}
Let $\G$ be a finitely presented group with Dehn function $\d_1(t) \lesssim t^N$ for some $N \in \re_{>0}$. Let $V$ be an isometric representation of $\G$ on some Banach space. The comparison map:
 \begin{equation*}
     H_{\mathrm{pol}: N+1 \to N}^2(\G, V) \to H^2(\G, V)
 \end{equation*}
is an isomorphism.
\end{cor}

Another example of interest are hyperbolic groups, since all of their homological filling functions are linear \cite[Theorem 4]{lang-hyp-filling-functions}, we obtain the following result.

\begin{cor}\label{Corollary: Comparison theorem for hyperbolic groups}
    Let $\G$ be a discrete hyperbolic group. Let $V$ be an isometric representation of $\G$ on some Banach space. The comparison map:
 \begin{equation*}
     H_{\mathrm{pol}: k \to 1}^k(\G, V) \to H^k(\G, V)
 \end{equation*}
is an isomorphism.
\end{cor}

\begin{proof}
    Since all the homological filling functions of $\G$ are linear  \cite[Theorem 4]{lang-hyp-filling-functions}, we have $N_k = \prod_{j= 2}^k \deg (\mathrm{cFV_X^j}) = 1$. We can compute $\a_j(1) = \a_{j-1}(1) +1$ and obtain $\a_j(1) = j+1$. inductively. Hence $M_{k-1}= \a_{k-1}(N_k) = \a_{k-1}(1) = k$.
\end{proof}

In fact the previous example is somewhat special. Whenever $\G$ is not a hyperbolic group, we always need to take coboundaries coming from cochains growing at polynomial speed at most the square of the growth of their differentials.

\begin{cor}\label{Corollary: Comparison theorem in other cases}
    Let $\G$ be a discrete group of type $F_{k+1}$ and $k \geq 3$ (the case $k = 2$ is Corollary \ref{Corollary: Comparison theorem in degree 2}). Suppose that $\G$ is not hyperbolic and that all the filling functions $\mathrm{cFV}_\G^j$ are polynomially bounded for $2 \leq j \leq k$, set $N_k = \prod_{j=2}^k \deg (\mathrm{cFV}_\G^j)$. Let $V$ be an isometric representation of $\G$ on some Banach space. For every $N > N_k$ and $M >  N^2 + (k-2)N$, the comparison map:
 \begin{equation*}
     H_{\mathrm{pol}:M \to N}^k(\G, V) \to H^k(\G, V)
 \end{equation*}
is an isomorphism.
\end{cor}

\begin{proof}
    Since $\G$ is not hyperbolic, we have $N > N_k \geq \deg \mathrm{cFV}_\G^2 \geq 2$ for $k \geq 2$ and hence $\a_2(N) = \max \{ (N + 1)N, 2 N +1 \} = (N + 1)N$. Now inductively we see that $\a_j(N) = N^2 + (j-1)N$ for $j \leq k$.
\end{proof}

In particular, if $\G$ is a $CAT(0)$ group acting properly and cocompactly on some $CAT(0)$-space $X$, for $k \geq 2$ we have $N_k = \prod_{j=2}^k \deg \mathrm{cFV}^j_X = \prod_{j=2}^k \frac{j}{j-1} = k$ \cite{wenger-short-proof}, so we have $\a_{k-1}(N_k) = k^2 + (k-2) k = 2k^2 - 2k$. The same computation applies to any non-cocompact arithmetic lattice $\G$ in a semisimple Lie group $G$ with $ \mathrm{rank}_\re G \geq 3$, for $k < \mathrm{rank}_\re G$ \cite{leuzinger-young}.

\section{Induction and transfer}\label{Section: Induction and transfer}

This section introduces induction of isometric representations on Banach spaces and general techniques for inducing quantitative polynomial cohomology of these representations. We also define one of the main technical tools of this paper, namely a transfer operator for quantitative polynomial cohomology, give sufficient conditions for its existence and prove its fundamental properties.

\subsection{$\textrm L^p$-induction of representations through a ME-coupling}

Let $\G$ and $\L$ be two countable ME groups, with ME-coupling $(\Omega, X_\L, X_\G, m)$, and associated cocycles $c_\G : \L \times X_\G \to \G$ and $c_\L : \G \times X_\L \to \L$. 

\begin{defn}
    Let $(\rho, V)$ be an isometric representation of $\L$ on some separable Banach space $V$. The \textit{$\textrm L^p$-induced representation} $I^p(\rho)$ of $\rho$ is the representation on the separable Banach space $I^p(V) = L^p(X_\L, V)$ of Bochner integrable functions, that is, the space of measurable maps $F : X_\L \to V$ such that:
    \begin{equation*}
        ||F||_{L^p(X_\L, V)}^p = \int_{X_\L} ||F(x)||_{V}^p \, \mathrm{d}m(x) < \infty,
    \end{equation*}
    where the \textit{induced $\G$-action} of $\g \in \G$ on $F \in L^p(X_\L, V)$ is given by:
    \begin{equation*}
        (I\rho(\g) F)(x) = \rho(\a_\L(\g, x)) F (\g^{-1} \cdot x).
    \end{equation*}
\end{defn}

This defines an isometric $\G$-action on $ L^p(X_\L, V)$.
Indeed, since the action of $\G$ on $(X_\L, m_{|X_{\L}})$ is measure preserving and $\rho$ is isometric we have:
\begin{align*}
    ||I \rho (\g) F||_{ L^p(X_\L, V)}^p & = \int_{X_\L} || \rho(\a_\L(\g, x)) F (\g^{-1} \cdot x)||_{V}^p \, \mathrm{d}m(x) \\
    & = \int_{X_\L} || \rho(c_\L(\g, y)) F (y)||_{V}^p \, \mathrm{d}m(y) = ||F||_{ L^p(X_\L,V)}^p.
\end{align*}

\begin{rem}
    No need to ask $\textrm L^p$-ME to define the induced representation nor the induced action. However, we will require integrability to induce polynomially growing cocycles into this representation.
\end{rem}

\subsection{Induction of quantitative polynomial cohomology}

\begin{prop}\label{Induction on quantitative polynomial cohomology}
    Fix $N \geq 0$ and $p \geq 1$. Let $\L$ and $\G$ be two finitely generated groups. Suppose that $\L$ and $\G$ are $\textrm L^r$-ME for some $r \geq pN$ and let $(\Omega, X_\L, X_\G,  m)$ be an $L^{r}$-integrable ME-coupling between $\L$ and $\G$ with associated cocycles $c_\L : \G \times X_\L \to \L$ and $c_\G : \L \times X_\G \to \G$. Then, for every isometric representation $(\rho, V)$ of $\L$ on some separable Banach space, the induction map
    \begin{align*}
        I : C_{\mathrm{pol} \leq N}(\L^{k+1}, V)^\L \to C_{\mathrm{pol} \leq N}(\G^{k+1}, L^p(X_\L, V))^\G\\
        Ib(\g_0, \ldots, \g_k)(x) = b(\a_\L(\g_0, x), \ldots, \a_\L(\g_k, x) ).
    \end{align*}
    is well-defined. It is a continuous linear map that commutes with differentials and for any $M$ such that $r/p \geq  M \geq N$, it induces continuous linear maps:
    \begin{align*}
        H_{\mathrm{pol}: M\to  N}^k(\L, \rho) \to H_{\mathrm{pol}: M\to  N}^k(\G, I^p (\rho)), \\
        \overline{H}_{\mathrm{pol}: M\to  N}^k(\L, \rho) \to \overline{H}_{\mathrm{pol}: M\to N}^k(\G, I^p(\rho)).
    \end{align*}
\end{prop}

\begin{proof}
    Let $b \in C_{\mathrm{pol} \leq N}(\L^{k+1}, V) $ and $\g_0, \ldots, \g_k \in \G$. We have:
    \begin{align*}
        ||Ib(\g_0, \ldots, \g_k)||_{ L^p(X_\L, V)}^p &= \int_{X_\L} ||b(\a_\L(\g_0, x), \ldots, \a_\L(\g_k, x) )||_V^p \, \mathrm{d}m (x) \\
        &\leq C_1  \int_{X_\L} (1 + |\a_\L(\g_0, x)|_{S_\L} + \ldots + |\a_\L(\g_k, x)|_{S_\L} )^{pN} \, \mathrm{d}m (x) \\
        &\leq C_1 \sum_{i= 0}^k \int_{X_\L}  |\a_\L(\g_i, x)|_{S_\L}^{pN} \, \mathrm{d}m (x) \\
        &\leq C_1 \sum_{i= 0}^k \int_{X_\L}  |\a_\L(\g_i, x)|_{S_\L}^r \, \mathrm{d}m (x) < \infty.
    \end{align*}

    Hence $Ib(\g_0, \ldots, \g_k) \in L^p(X_\L, V)$ and so $Ib \in  C(\G^{k+1}, L^p(X_\L, V))$. For $r>0$ we set $M_r = \sup_{s \in S_{\G}} \int_{X_\L}  |\a_\L(s, x)|_{S_\L}^r \, \mathrm{d}m (x) $. From this inequality and Lemma \ref{Growth of integrals is sublinear} we can see that:
    \begin{equation*}
        ||Ib(\g_0, \ldots, \g_k)||_{ L^p(X_\L, V)}^p 
         \leq C_1 \sum_{i= 0}^k \int_{X_\L}  |\a_\L(\g_i, x)|_{S_\L}^{pN} \, \mathrm{d}m (x) \leq C_2 M_{pN} \sum_{i = 0}^k |\g_i|_{S_\G}^{pN}.
    \end{equation*}
    Hence the subadditivity of $t \mapsto t^{1/p}$ implies: 
    \begin{equation*}
        ||Ib(\g_0, \ldots, \g_k)||_{ L^p(X_\L, V)} \leq C_3 M_{pN}^{1/p} \sum_{i = 0}^k |\g_i|_{S_\G}^{N}.
    \end{equation*}
    and so $Ib \in C_{\mathrm{pol} \leq N}(\G^{k+1}, L^p(X_\L, V))$.

    The fact that $Ib$ is $\G$-equivariant, and that the resulting map $I$ commutes with differentials are standard \cite[Section 3.2]{shalom2004harmonic}. The continuity of $I$ in the topology of uniform convergence on compact subsets (which here is just the topology of pointwise convergence) follows from the dominated convergence theorem.
    
    The map $I$ induces maps on the quotient spaces for $r/p \geq M \geq N$ because as $pM \leq r$, the map $I:  C_{\mathrm{pol} \leq M}(\L^{k}, V) \to C_{\mathrm{pol} \leq M}(\G^{k}, L^p(X_\L, V))$ is also well-defined, $\G$-equivariant, continuous and commutes with differentials.
    
\end{proof}

\subsection{Transfer operator and quantitative polynomial cohomology}

In this section we define a transfer operator as in \cite[3.2.1]{shalom2004harmonic} for polynomial cohomology and reduced polynomial cohomology for ME-couplings with sufficient integrability. 

We keep the notation from the previous section, that is, we let $\L$ and $\G$ be two finitely generated ME groups, $(\Omega, X_\L, X_\G, m)$ a ME-coupling between $\L$ and $\G$ with associated cocycles $c_\L$ and $c_\G$.
For $\l \in \L$ and $\g \in \G$ define:
\begin{equation*}
    A_\l(\g) := \{ x \in X_\G, \, \a_\G(\l, x) = \g \} = \a_\G(\l, \cdot)^{-1}(\{ \g \}).
\end{equation*}
Notice that for each $\l \in \L$, $(A_\l(\g))_{\g \in \G}$ defines a partition of $X_\G$.

We also define for $\eta = (\l_0, \ldots, \l_{k}) \in \L^{k+1}$ and $\xi = (\g_{0}, \ldots, \g_{k}) \in \G^{k+1}$:
\begin{equation*}
     A_\eta(\xi) :=  A_{\l_0}(\g_0) \cap \ldots \cap   A_{\l_{k}}(\g_{k}).
\end{equation*} 
If we consider the product ME-coupling $(\Omega^{k+1}, X_\L^{k+1}, X_\G^{k+1}, m^{\otimes(k+1)})$ between $\L^{k+1}$ and $\G^{k+1}$, with associated cocycles $c_{\L^{k+1}}(\xi, \overline{x}) = (c_\L (\g_0, x_0), \ldots,c_\L (\g_k, x_k) )$ and $c_{\G^{k+1}}(\eta , \overline{x}) = (c_\G (\l_0, x_0), \ldots,c_\G (\l_k, x_k) )$, we see that the sets $ A_\eta(\xi)$ can be seen as fibres for these cocycles after restriction to the diagonal via the diagonal inclusion $\diag: X_\G \to X_\G^{k+1}, x \mapsto (x, \ldots, x)$. We have:
\begin{equation*}
    A_\eta(\xi) =   A_{\l_0}(\g_0) \cap \ldots \cap   A_{\l_{k}}(\g_{k}) = (\a_{\G^{k+1}}(\eta, \cdot)\circ \diag) ^{-1}(\{ \xi\}).
\end{equation*}
Hence for each $\eta \in \L^{k+1}$, another partition of $X_\G$ is given by:
\begin{equation*}
    X_\G = \bigsqcup_{\xi \in \G^{k+1}} A_\eta(\xi).
\end{equation*}

The cocycle relation for $c_\G$ (and for $\a_\G$) gives an important relation satisfied by the subsets $A_\l(\g)$.

\begin{lem}\label{Cocycle relation for partitions of X_G}
For every $\l, \l' \in \L$ and $\g, \g' \in \G$ we have:
\begin{equation*}
    A_\l(\g) \cap A_{\l \l'}(\g') = A_\l(\g) \cap \l \cdot A_{\l'}(\g^{-1} \g').
\end{equation*}
Similarly, for every $\eta \in \L^{k+1}, \xi \in \G^{k+1}, \l \in \L$ and $\g \in \G$ we have:
\begin{equation*}
    A_\l(\g) \cap A_{\l \eta}(\xi) = A_\l(\g) \cap \l \cdot A_{\eta}(\g^{-1} \xi).
\end{equation*}
\end{lem}
\begin{proof}
    For $\l, \l' \in \L$, the cocycle relation for $\a_\G$ reads as:
    \begin{equation*}
        \a_\G(\l \l', x) = \a_\G(\l, x) \a_\G(\l', \l^{-1} \cdot x).
    \end{equation*}
    The condition $x\in  A_\l(\g) \cap A_{\l \l'}(\g')$ means that $\a_\G(\l, x) = \g$ and $\a_\G(\l \l', x) = \g'$. Hence $\a_\G(\l', \l^{-1} \cdot x) = \g^{-1} \g '$ and so $x \in \l \cdot A_{\l'}(\g^{-1} \g')$. The reverse inclusion is shown in the same way. The second statement is proven by applying the first statement to each component.
\end{proof}

Let $1 \leq p < \infty$ and $(\rho, V)$ denote some isometric representation of $\L$ on some separable Banach space. Let $B \in  C^k(\G, I^p(V))$, we define (if it makes sense) for $\eta = (\l_0, \ldots, \l_{k}) \in \L^{k+1}$:
\begin{equation*}
    TB(\eta) := \sum_{\xi \in \G^{k+1}} \int_{ A_{\eta}(\xi)} B(\xi )(x) \mathrm{d} m (x).
\end{equation*}

Since for every $\xi \in \G^{k+1}$ we have $B(\xi ) \in I^p(V) = L^p(X_\L, V)$ and $X_\L$ has finite measure, each of the integrals $\int_{ A_{\eta}(\xi)} B(\xi )(x) \mathrm{d} m (x)$ exist, but it is not automatic that the series defining $TB(\eta)$ converge. The following Theorem gives a sufficient condition guaranteeing the existence of $TB(\eta)$.

\begin{thm}\label{Higher transfer operator is well defined} Let $1<p,q<\infty$ such that $\frac{1}{p} + \frac{1}{q} = 1$ and $k \in \na_{\geq 2}$.
Suppose that $X_\G \subseteq X_\L$ and that there exists $\varepsilon >0$ such that for every $\eta \in \{ 1_\L\} \times \L^k$ we have:
\begin{equation*}
    \int_{X_\G} \#S_{\G^{k}}( |\a_{\G^{k+1}}(\eta, x) | )^{\frac{q}{p}} |\a_{\G^{k+1}}(\eta, x) |^{q(N+1) - 1 + \varepsilon} \mathrm{d} m (x) <  \infty.
\end{equation*}
Then the series defining the map $TB(\l_0, \ldots, \l_{k})$ as above converges absolutely for every $\l_0, \ldots, \l_{k} \in \L$. The map $T: B \to TB$ defines a linear map \begin{equation*}
    T : C^k_{\mathrm{pol} \leq N} (\G,\textrm L^p(X, V))^\G \to C^k(\L , V)^\L 
\end{equation*} which induces continuous linear maps for every $M\geq N$:
\begin{align*}
    & H^k_{\mathrm{pol}:M \to N} (\G, I^p(\rho)) \to  H^k(\L , \rho)  , \\
    & \overline{H}^k_{\mathrm{pol}:M \to N} (\G,I^p(\rho)) \to \overline{H}^k(\L , \rho) .
\end{align*}
If moreover the coupling is $L^{pN}$-integrable on $\L$, then we have $T \circ I = m (X_\G) \Id$. 

\end{thm}

\begin{proof}
Let $B \in   C^k_{\mathrm{pol} \leq N}(\G, I^p(V))$.  We first set for $ \xi = (\g_{1}, \ldots, \g_{k+1}) \in \G^{k+1}$, 
\begin{equation*}
    |\xi| = |\g_{1}| +  \ldots + |\g_{k+1}|.
\end{equation*}
The polynomial growth condition of $B$ gets rewritten as, for every $\xi \in \G^{k+1}$:
\begin{equation*}
     ||B(\xi )||_{L^p( X_\L, V)} \leq C (1 + |\xi|)^N.
\end{equation*}

Let $\eta = (\l_0, \ldots, \l_{k}) \in \L^{k+1}$. Since the cocycle $TB$ that we want to define will be $\L$-equivariant, we will assume that $\l_{0} = 1_\L$. Since $X_\G \subseteq X_\L$, we have $A_{1_\L}(\g) = \emptyset$ for $\g \in \G \setminus \{ 1_\G \}$. This means that we can take out the first variable in the sum and only have to look at the absolute convergence of 
\begin{equation*}
    \sum_{\xi = (\g_1, \ldots, \g_k) \in \G^{k}} \int_{ A_{(1, \l_1, \ldots, \l_{k})}(1,\g_1, \ldots, \g_k)} B(1,\g_1, \ldots, \g_k )(x) \mathrm{d} m (x).
\end{equation*}
We first use Hölder's inequality and the fact that $B$ has polynomial growth to obtain the following inequalities:
\begin{align*}
        \sum_{\xi \in \{ 1 \} \times \G^{k}} || \int_{ A_\eta(\xi) } B(\xi )(x) \mathrm{d} m (x) || & \leq \sum_{\xi \in\{ 1 \} \times \G^{k}} ||B(\xi )||_{\textrm L^p( A_\eta(\xi), V)} \, m(A_\eta(\xi))^{1/q} \\ 
        & \leq \sum_{\xi \in \{ 1 \} \times \G^{k}} ||B(\xi )||_{\textrm L^p( X_\L, V)} \, m(A_\eta(\xi))^{1/q} \\ 
        & \leq C  \sum_{\xi \in \{ 1 \} \times \G^{k}} (1 + |\xi|)^N \,  m(A_\eta(\xi))^{1/q}
    \end{align*}

It remains to estimate the convergence of the series 
\begin{equation*}
    \sum_{\xi \in \{ 1 \} \times \G^{k}}  |\xi| ^N \,  m(A_\eta(\xi))^{1/q}.
\end{equation*} 
Again, this relies on Hölder's inequality. For any $\e>0$ we have:
\begin{align*}
    & \sum_{\xi \in \{ 1 \} \times \G^{k}} |\xi| ^N \,  m(A_\eta(\xi))^{1/q} \\ 
    &= \sum_{\xi \in \{ 1 \} \times \G^{k}} \big(|\xi|^{- \frac{1 + \e}{p}} \#S_{\G^{k}}(|\xi|)^{-1/p} \big)  \big(\#S_{\G^{k}}(|\xi|)^{1/p}  |\xi|^{N + \frac{1 + \e}{p}} \, m(A_\eta(\xi))^{1/q} \big) \\
    & \leq \Big( \sum_{\xi \in \{ 1 \} \times \G^{k}} |\xi|^{-(1 + \e)} \#S_{\G^{k}}(|\xi|)^{-1} \Big)^{\frac{1}{p}} \Big( \sum_{\xi \in \{ 1 \} \times \G^{k}} \#S_{\G^{k}}(|\xi|)^{\frac{q}{p}}  |\xi|^{ q N + q \frac{1 + \e}{p}} \, m(A_\eta(\xi)) \Big)^{\frac{1}{q}}  \\
    = & \Big(\sum_{n \in \na} n^{- (1 + \e)} \Big)^{\frac{1}{p}} \Big( \sum_{\xi \in \{ 1 \} \times \G^{k}} \#S_{\G^{k}}(|\xi|)^{\frac{q}{p}}  |\xi|^{ q N + q \frac{1 + \e}{p}} \, m(A_\eta(\xi)) \Big)^{\frac{1}{q}}  \\
     = & \Big(\sum_{n \in \na} n^{- (1 + \e)} \Big)^{\frac{1}{p}} \Big( \int_{X_\G} \#S_{\G^{k}}( |\a_{\G^{k+1}}(\eta, x) | )^{\frac{q}{p}} |\a_{\G^{k+1}}(\eta, x) |^{(N+1)q - 1 + \frac{\e q }{p}} \mathrm{d} m (x) \Big)^{\frac{1}{q}} \\
     =& < \infty.
\end{align*}

Hence the series defining $TB(\eta)$ converge for $\eta \in \{1\} \times\L^{k}$. We will now show that $\L$-equivariance of the formula defining $TB(\eta$), which in turn will allow us extend the definition of $TB(\eta)$ to all of $\L^{k+1}$.
Let $\l \in \L, \eta = (1 ,\l_1, \ldots, \l_k) \in \{1\} \times\L^{k}$ and $B \in C_{\mathrm{pol \leq N}}^k(\G,I^p(\rho))$. Using the $I\rho (\g)$-equivariance of the cochain $B$, Lemma \ref{Cocycle relation for partitions of X_G}, the fact that the measure $m|_{X_\G}$ is $\L$-invariant and that the sets $(\l^{-1} \cdot A_\l(\g))_{\g \in \G}$ for a partition of $X_\G$, we have:
\begin{align*}
    TB( \l \eta)
    &= \sum_{\g \in \G, \xi' \in \G^{k}} \int_{A_\l(\g) \cap A_{\l \eta'}(\xi')} B(\g , \xi')(x)  \mathrm{d} m (x) \\
    & =\sum_{\g \in \G, \xi' \in \G^{k}} \int_{A_\l(\g) \cap A_{\l \eta'}(\xi')} (I \rho(\g) B(1 , \g^{-1} \xi'))(x)  \mathrm{d} m (x) \\
    &=\sum_{\g \in \G, \xi' \in \G^{k}} \int_{A_\l(\g) \cap A_{\l \eta'}(\xi')} \rho(\l) B(1 , \g^{-1} \xi')(\l ^{-1} \cdot x)  \mathrm{d} m (x) \\
    &=\sum_{\g \in \G, \xi' \in \G^{k}} \int_{A_\l(\g) \cap \l \cdot A_{\eta'}(\g^{-1}\xi')} \rho(\l) B(1 , \g^{-1} \xi')(\l ^{-1} \cdot x)  \mathrm{d} m (x) \\
    &=\sum_{\g \in \G, \xi' \in \G^{k}} \int_{\l^{-1}\cdot A_\l(\g) \cap  A_{\eta'}(\g^{-1}\xi')} \rho(\l) B(1 , \g^{-1} \xi')(y)  \mathrm{d} m (y) \\
     &=\sum_{\g \in \G, \xi' \in \G^{k}} \int_{\l^{-1}\cdot A_\l(\g) \cap  A_{\eta'}(\xi')} \rho(\l) B(1 , \xi')(y)  \mathrm{d} m (y) \\
     &=\sum_{\xi' \in \G^{k}} \int_{A_{\eta'}(\xi')} \rho(\l) B(1 , \xi')(y)  \mathrm{d} m (y) = \rho(\l) TB(\eta).
\end{align*}
Hence if the series defining $TB(\eta)$ converge absolutely, $TB(\l\eta)$ is also well defined, and the resulting cochain $TB$ is $\L$-equivariant.

The fact that the operator $T$, when well-defined, commutes with differentials is pretty formal. For $\eta =(\l_0, \ldots, \l_k) \in \L^{k+1}$, we write $\eta_i = (\l_0, \ldots, \hat{\l_i} , \ldots, \l_k ) \in \L^k$ for $0 \leq i \leq k$, similarly for elements $\G^{k+1}$. Let $\eta =(\l_0, \ldots, \l_k) \in \L^{k+1}$ and $B \in C^{k-1}_{\mathrm{pol} \leq N}(\G,I^p(V))$. We have:
\begin{align*}
    TdB (\l_0, \ldots, \l_k) & = \sum_{\xi = (\g_0, \ldots, \g_k) \in \G^{k+1}} \int_{A_{\eta}(\xi)} \sum_{i= 0}^k (-1)^{i} B(\xi_i ) (x) \mathrm{d} m (x) \\
    & = \sum_{i= 0}^k (-1)^{i} \sum_{\xi_i = (\g_0, \ldots, \hat{\g_i}, \ldots,  \g_k) \in \G^{k}} \sum_{\g_i \in \G} \int_{A_{\l_i}(\g_i) \cap A_{\eta_i}(\xi_i)} B(\xi_i ) (x) \mathrm{d} m (x) \\
    & = \sum_{i= 0}^k (-1)^{i} \sum_{\xi_i = (\g_0, \ldots, \hat{\g_i}, \ldots,  \g_k) \in \G^{k}} \int_{ A_{\eta_i}(\xi_i)} B(\xi_i ) (x) \mathrm{d} m (x) \\
    & = \sum_{i= 0}^k (-1)^{i} TB(\xi_i) = dTB (\l_0, \ldots, \l_k).
\end{align*}

Continuity of the map $T$ in the topology of uniform convergence on compact subsets follows from the dominated convergence theorem.

Finally, the last property uses that $X_\G \subseteq X_\L$ in order to have $\a_\L (\a_\G(\l, x),x) = \l$ for every $\l \in \L$ and $x \in X_\G$. For $c \in C_{\mathrm{pol} \leq N}^k(\G, \rho)$ and $\eta= (\l_0, \ldots, \l_k)\in \L^{k+1}$, we have:
\begin{align*}
    T I c(\eta) = \sum_{\xi\in \G^{k+1}} \int_{A_\eta (\xi)} Ic(\xi) (x) \mathrm{d} m (x)
\end{align*}
and for $x \in A_{\eta}(\xi)$, we have $Ic(\xi) (x) = c(\eta)$. Since $(A_\eta (\xi))_{\xi\in \G^{k+1}}$ is a partition of $X_\G$, we obtain $T I c(\eta) = m(X_\G) c(\eta)$.
\end{proof}

\section{Applications for virtually nilpotent groups}\label{Section: Applications for virtually nilpotent groups}

The goal of this section is to show Theorem \ref{Intro Theorem: Quantitative invariance of Betti numbers of nilpotent groups}, that is, invariance of Betti numbers of virtually nilpotent groups by mutually cobounded $\textrm L^p$-ME when $p$ is large enough, and to give examples of nilpotent groups with different Betti numbers.

We will use three general properties of virtually nilpotent groups.

The first property we need, which turns out to characterize virtually nilpotent groups, is polynomial growth (in the sense that there exists $d \in \re_{>0}$ such that $\#B_{\G,S}(n) := \{ \g \in \G, |\g|_S \leq n  \} \lesssim n^d$), which is guaranteed by Gromov's polynomial growth theorem. 
\begin{thm}(Gromov \cite{gromov-polynomial-growth})\label{Polynomial growth theorem}
    A finitely generated group $(\G, S)$ has polynomial growth if and only if $\G$ is virtually nilpotent.
\end{thm}

The second property we need is polynomiality of combinatorial filling functions.

\begin{thm}\label{Polynomiality of filling functions of nilpotent groups}
    Let $\G$ be a finitely generated nilpotent group. Then for every $d \in \na_{\geq 2}$, $\G$ is of type $F_d$ and the combinatorial filling function $\mathrm{cFV}^d_\G$ is polynomially bounded.
\end{thm} 

Finiteness properties of finitely generated nilpotent groups come from the fact that they are cocompact lattices in their Mal'cev completions \cite{malcev}.
It is hard to single out one reference for the polynomiality of filling functions of nilpotent groups. This is stated in \cite[p. 56 or depending on edition p. 82]{gromov}.
The book \cite{varopoulos-potential} allows to see it in different ways. For instance, in the book's terminology, nilpotent Lie groups are NC \cite[Section 2.2.1]{varopoulos-potential}, and if a group is NC then it is polynomially retractable \cite[Theorem 7.10]{varopoulos-potential} and being polynomially retractable implies having the polynomial filling property \cite[Definition 7.12]{varopoulos-potential}, where the latter means that the more classical Lipschitz versions of the filling functions $\mathrm{FV}_\G^k$ are polynomially bounded for every $k \in \na_{\geq 2}$. The Federer-Fleming deformation technique shows that $\mathrm{cFV}_\G^k \sim \mathrm{FV}_\G^k$ (see for instance \cite[Theorem 6.6]{bader-sauer}).

The third property that we need is the higher analogue of Shalom's property $H_T$ \cite[p. 125]{shalom2004harmonic}. To our knowledge, this property was first shown for nilpotent groups by Blanc \cite{blanc}.

\begin{thm} (Blanc \cite[10.5]{blanc}, \cite[p. 243]{guichardet-livre}) \label{Higher property H_T for nilpotent groups}
  Let $G$ be a connected and simply connected nilpotent Lie group. Then for every $k \geq 1$ and for every continuous unitary representation $\pi$ without $G$-invariant vectors we have:
\begin{equation*}
    \overline{H}^k(G, \pi) = 0.
\end{equation*}
\end{thm}
By Shapiro's lemma for cocompact lattices \cite[3.2.2]{shalom2004harmonic}, this property holds for torsion-free finitely generated nilpotent groups. By Shapiro's lemma for finite index subgroups, it also holds for general finitely generated virtually nilpotent groups.

\subsection{Injectivity of induction}

We first specialize Theorem \ref{Higher transfer operator is well defined} to the case of virtually nilpotent groups, by introducing a more concrete integrability hypothesis. Theorem \ref{Intro Theorem: Injectivity of induction in polynomial cohomology} stated in the introduction is proven along the way (as it is just Theorem \ref{Higher transfer operator is well defined} with integrability conditions coming from nilpotent groups).

The following statement and its consequences will depend on the number $d_\G = \inf \{ d \in \re_{>0}, \#S_\G(n) \lesssim n^d \}$, where $S_\G (n) : = \{ \g \in \G, |\g|_S = n \}$ denotes the sphere of radius $n$ for the length function $| \cdot |$. In the introduction, we stated results using the constant $d(\G) = \inf \{ d \in \re_{>0}, \#B_\G(n) \lesssim n^d \}$, where $B_\G (n) : = \{ \g \in \G, |\g|_S \leq n \}$. These two differ at most by 1 as we have $d(\G) - 1 \leq d_\G \leq d(\G)$. Conjecturally we have $d_\G = d(\G) - 1$ \cite[8.2]{breuillard-le-donne} (see also \cite[8.3]{breuillard-le-donne}).

\begin{thm}\label{Injectivity of induction for nilpotent groups and arbitrary representation} Suppose that $\L$ and $\G$ are finitely generated virtually nilpotent groups and let $d_\G = \inf \{ d \in \re_{>0}, \#S_\G(n) \lesssim n^d \} < \infty$.
    Let $1<p,q<\infty$ such that $\frac{1}{p} + \frac{1}{q} = 1$ and $k \in \na, k \geq 2$. Let $N_{\L, k} = \prod_{i=2}^k \deg \mathrm{cFV_\L^i}, N_{\G,k}= \prod_{i=2}^k \deg \mathrm{cFV_\G^i}$ and $N = \max \{N_{\L, k}, N_{\G, k} \}$. Let $M> \max\{ \a_{\L,k-1}(N), \a_{\G,k-1}(N) \}$, where $\a_{\L,k-1}(N)$ and $\a_{\G,k-1}(N)$ are the constants appearing in Theorem \ref{Quantitative isomorphism of polynomial cohom onto ordinary cohom} when applied for both groups $\L$ and $\G$.
    
Suppose that $X_\G \subseteq X_\L$, the coupling is $L^{pM}$-integrable on $\L$ and that the coupling is $L^r$-integrable on $\G$ for some \begin{equation*}
    r >  \frac{d_\G k q}{p} + (M+1) q - 1.
\end{equation*}
Let $(\rho, V)$ be an isometric representation of $\L$ on some separable Banach space $V$. Then the induction map $I$ induces injective continuous maps
\begin{align*}
    H_{\mathrm{pol:}M \to N}^k(\L, \rho) &\xhookrightarrow{} H_{\mathrm{pol:}M \to N}^k(\G, I^p(\rho)), \\
    \overline{H}_{\mathrm{pol:}M \to N}^k(\L, \rho) &\xhookrightarrow{} \overline{H}_{\mathrm{pol:}M \to N}^k(\G, I^p(\rho)), \\
    H^k(\L, \rho) &\xhookrightarrow{} H^k(\G, I^p(\rho)), \\
    \overline{H}^k(\L, \rho) &\xhookrightarrow{} \overline{H}^k(\G,I^p(\rho)).
\end{align*}
\end{thm}

\begin{proof}
    Let $\xi= (\g_0, \ldots, \g_k) \in \{1_\G\} \times \G^{k}$. Theorem \ref{Polynomial growth theorem} says that $d_\G < \infty$, let $d< \infty$ be such that $\#S_\G(n) \lesssim n^d$, so that $\#S_{\G^k}(n) \lesssim n^{dk}$. Since $r > d_\G k \frac{q}{p} + (M+1) q - 1$, we can always choose $\varepsilon>0$ small and $d$ close to $d_\G$ so that we have $r- \varepsilon > d k \frac{q}{p} + (M+1) q - 1$. We have:
    \begin{align*}
    & \int_{X_\G} \#S_{\G^{k}}( |\a_{\G^{k+1}}(\xi, x) | )^{\frac{q}{p}} |\a_{\G^{k+1}}(\xi, x) |^{q(M+1) - 1 + \varepsilon} \mathrm{d} m (x) \\
    & \leq C_1  \int_{X_\G} |\a_{\G^{k+1}}(\xi, x) |^{\frac{d k q}{p} + q(M+1) - 1 + \varepsilon} \mathrm{d} m (x) \\ 
    &\leq C_1  \int_{X_\G} |\a_{\G^{k+1}}(\xi, x) |^{r} \mathrm{d} m (x) \\
     &\leq C_2  \sum_{i= 0}^{k} \int_{X_\G} |\a_{\G}(\g_i, x) |^{r} \mathrm{d} m (x).
\end{align*}
The last integral being finite (the coupling is $\textrm L^r$-integrable on $\G$), we can apply Theorem \ref{Higher transfer operator is well defined}, which says that the transfer map:
\begin{equation*}
        C_{\mathrm{pol} \leq D} (\G^{i+1}, L^p(X_\L, V))^\G\xrightarrow{T} C(\L^{i+1}, V)^\L
    \end{equation*}
    is well-defined for any $D \leq M$ and $i \leq k$.

    Since moreover the coupling is $\textrm L^{pM}$-integrable on $\L$, Proposition \ref{Induction on quantitative polynomial cohomology} says that the induction and transfer maps
    \begin{equation*}
        C_{\mathrm{pol} \leq D}(\L^{i+1}, V)^\L \xrightarrow{I} C_{\mathrm{pol} \leq D} (\G^{i+1}, L^p(X_\L, V))^\G\xrightarrow{T} C(\L^{i+1}, V)^\L
    \end{equation*}
    are well-defined and continuous for $i \leq k$ and any $D \leq M$. Since $T \circ I = m(X_\G) \,\Id$, we obtain induced maps in cohomology:
    \begin{equation*}
        H_{\mathrm{pol:}M \to N}^k(\L, \rho) \xrightarrow{I} H_{\mathrm{pol:}M \to N}^k(\G, I^p(\rho)) \xrightarrow{T} H^k(\L, \rho),
    \end{equation*}
    such that composition of these two maps is the natural comparison map (and this passes to reduced cohomology). Because of Theorem \ref{Polynomiality of filling functions of nilpotent groups}, we can apply Theorem \ref{Intro theorem: Quantitative isomorphism of polynomial cohom onto ordinary cohom} to the group $\L$ and obtain that $H_{\mathrm{pol:}M \to N}^k(\L, \rho) \to H^k(\L, \rho)$ is an isomorphism. This implies that the induction map on polynomial cohomology:
    \begin{equation*}
         H_{\mathrm{pol:}M \to N}^k(\L, \rho) \xrightarrow{I} H_{\mathrm{pol:}M \to N}^k(\G, I^p(\rho))
    \end{equation*}
    is injective (and also on reduced polynomial cohomology). By our choice of $M$ and $N$, Theorem \ref{Intro theorem: Quantitative isomorphism of polynomial cohom onto ordinary cohom} can also be applied to the group $\G$, and hence we obtain continuous injective maps:

\begin{center}
    \begin{tikzcd}[scale=0.7em]
	H_{\mathrm{pol:}M \to N}^k(\L, \rho) & H_{\mathrm{pol:}M \to N}^k(\G, I^p(\rho)) \\
	H^k(\L, \rho) & H^k(\G, I^p(\rho)) 
	\arrow["I", from=1-1, to=1-2]
	\arrow["{\simeq}"', from=1-1, to=2-1]
	\arrow["{\simeq}"', from=1-2, to=2-2]
	\arrow["I"', from=2-1, to=2-2]
\end{tikzcd}
\end{center}
    which also hold in reduced cohomology.
\end{proof}

\subsection{Betti numbers of nilpotent groups and $\textrm L^p$-ME}

The next Corollary is obtained by specializing the previous Theorem to the case of the trivial representation. Its proof is essentially the same as \cite[4.1.1]{shalom2004harmonic}.

\begin{cor}\label{Corollary: estimate on Betti numbers of nilpotent groups}
    Suppose that $\L$ and $\G$ are finitely generated virtually nilpotent groups. Let $d_\G = \inf \{ d \in \re_{>0}, \#S_\G(n) \lesssim n^d \}$.
    Let $k \in \na_{\geq 2}$. Let $N_{\L, k} = \prod_{i=1}^k \deg \mathrm{cFV_\L^i}, N_{\G,k}= \prod_{i=1}^k \deg \mathrm{cFV_\G^i}$ and $N = \max \{N_{\L, k}, N_{\G, k} \}$.  Let $M > \max \{\a_{\L,k-1}(N),\a_{\G,k-1}(N)\}$ where $\a_{\L,k}(N)$ and  $\a_{\G,k}(N)$ are the constants appearing in Theorem \ref{Quantitative isomorphism of polynomial cohom onto ordinary cohom} when applied for both groups $\L$ and $\G$.
    
Suppose that $X_\G \subseteq X_\L$, the coupling is $L^{2M}$-integrable on $\L$ and that the coupling is $L^r$-integrable on $\G$ for some \begin{equation*}
    r > d_\G k + 2M + 1.
\end{equation*}
Then $\dim_\re H^k(\L, \re) \leq \dim_\re H^k(\G, \re)$.
\end{cor}

\begin{proof}
    We are under the same conditions as in Theorem \ref{Injectivity of induction for nilpotent groups and arbitrary representation} for $p = q = 2$. We apply this Theorem for the trivial representation. The induced representation is the Koopman representation of $\G$ on $L^2(X_\L)$. Since the measure $m$ is finite and $\G$-invariant on $X_\L$, Krein-Milman theorem implies that there exists a $\G$-ergodic finite measure on $X_\L$. Thus we may assume $m$ is $\G$-ergodic, so that $L^2(X_\L)^\G $ is the space of constant functions. If we denote by $L^2_0(X_\L)$ the subspace of functions with zero integral, we have that $L^2(X_\L) =  L^2(X_\L)^\G  \oplus L^2_0(X_\L)$ and so $ L^2_0(X_\L)^\G = 0$. Theorem \ref{Injectivity of induction for nilpotent groups and arbitrary representation} says that we have an injective map:
    \begin{equation*}
         \overline{H}^k(\L, \re) \to \overline{H}^k(\G, L^2 (X_\L)) = \overline{H}^k(\G, \re) \oplus \overline{H}^k(\G, L^2_0 (X_\L)).
    \end{equation*}

    Since $\G$ is nilpotent, Theorem \ref{Higher property H_T for nilpotent groups} says that $\overline{H}^k(\G,L^2_0 (X_\L)) = 0$. The spaces $H^k(\L, \re)$ and $H^k(\G, \re)$ are automatically Hausdorff as they are finite dimensional real vector spaces and coincide with their reduced versions. Hence we have an injective linear map
    \begin{equation*}
         H^k(\L, \re) \to H^k(\G, \re),
    \end{equation*}
    from where we deduce the dimension estimate.
\end{proof}

\begin{proof}[Proof of Theorem \ref{Intro Theorem: Quantitative invariance of Betti numbers of nilpotent groups}]

Let $\L$ and $\G$ be finitely generated virtually nilpotent groups with the same associated Carnot group. Let $d$ be the degree of polynomial growth of both $\L$ and $\G$. Let $N = \max \{\prod_{i=2}^k \deg \mathrm{cFV_\L^i}, \prod_{i=2}^k \deg \mathrm{cFV_\G^i} \}$ and let $M > \a_{k-1}(N)$ where the functions $\a_i$ are defined in Theorem \ref{Intro theorem: Quantitative isomorphism of polynomial cohom onto ordinary cohom}. Since the groups $\L$ and $\G$ are not hyperbolic, we can choose $\a_{k-1} (N)$ to be $N^2 + (k-2) N$ when $k \geq 3$ (Corollary \ref{Corollary: Comparison theorem in other cases}) and $\a_{k-1} (N) = \a_1(N) = N+1$ when $k = 2$ (Corollary \ref{Corollary: Comparison theorem in degree 2}). Suppose that $\L$ and $\G$ admit a mutually cobounded $\textrm L^p$-ME coupling $(\Omega, X_\L, X_\G, m)$ for some $p > d k + 2M + 1$.

Lemma \ref{Putting a fundamental domain inside the other} says that we can change $\G$ by a finite product $\G \times K$ to obtain a coupling $(\overline{\Omega}, Y_\L, Y_{\G \times K}, m)$ satisfying $Y_\G \subseteq Y_\L$. Corollary \ref{Corollary: estimate on Betti numbers of nilpotent groups} says that 
\begin{equation*}
     \dim_\re H^k(\L, \re) \leq \dim_\re H^k(\G \times K, \re) = \dim_\re H^k(\G, \re)
\end{equation*}
where the latter equality follows from the classical Shapiro lemma for finite-index subgroups. After exchanging the roles of $\L$ and $\G$ and applying Lemma \ref{Putting a fundamental domain inside the other} and Corollary \ref{Corollary: estimate on Betti numbers of nilpotent groups} again, we obtain:
\begin{equation*}
     \dim_\re H^k(\L, \re) = \dim_\re H^k(\G, \re).
\end{equation*}
\end{proof}

\subsection{Nilpotent Lie algebras and gradings}

Every finitely generated virtually nilpotent group contains a finite index subgroup that is torsion-free and nilpotent \cite[Lemma 4.6]{raghunathan}.
Moreover, every finitely generated, torsion-free, nilpotent group $\G$ arises as a cocompact lattice in a simply connected nilpotent Lie group $G$, called the \textit{Mal'cev completion} of $\G$, whose corresponding Lie algebra $\ag$ has a basis with rational structural constants \cite{malcev} \cite[2.18]{raghunathan}. One can compute Betti numbers of $\G$ using Betti numbers of the Lie algebra $\ag$ (which are defined using Lie algebra cohomology). Indeed, since $\G$ is cocompact in $G$, we have that 
\begin{equation*}
    b_n(\G) = \dim_\re H^k(\G, \re) = \dim_\re H^k(G, \re)  = b_n(G)
\end{equation*} for every $n$ (this is can be seen using \cite[3.2.2]{shalom2004harmonic} and \cite[p. 243]{guichardet-livre}). Moreover, we also have $b_n(G) =  b_n(\ag)$ \cite[Chapitre II]{guichardet-livre}, where the latter means Betti numbers in the sense of Lie algebra cohomology.

The central series of a Lie algebra $\ag$ is the sequence of Lie subalgebras $C^1(\ag) = \ag$ and $C^i(\ag) = [\ag, C^{i-1}(\ag)]$ for $i \geq 2$. A Lie algebra $\ag$ is $s$-nilpotent (where $s \in \na_{\geq 2}$) if $C^{s+1}(\ag) = 0$.

We say that a Lie algebra $\ag$ is \textit{Carnot} or \textit{graded} if it admits a decomposition $\ag = \bigoplus_{i \geq 1} \mathfrak{m}_i$, where $C^{i}(\ag) = \bigoplus_{j\geq i} \mathfrak{m}_i$ and $[ \mathfrak{m}_i,  \mathfrak{m}_j] \subseteq  \mathfrak{m}_{i+j}$ for all $i, j$.

The graded Lie algebra $\mathrm{gr}(\ag)$ associated to a nilpotent Lie algebra $\ag$ is obtained by considering the quotients $\mathfrak{m}_i = C^i(\ag) / C^{i+1}(\ag)$ for $i \geq 1$. The graded Lie algebra of $\ag$ is the sum $\mathrm{gr}(\ag) := \bigoplus_{i \geq 1} \mathfrak{m}_i$ with the natural Lie algebra brackets induced by considering $[X_i, X_j]$ as an element in $\mathfrak{m}_{i+j}$ whenever $X_i \in \mathfrak{m}_{i}$ and $X_j \in \mathfrak{m}_{j}$.

Let $\G$ is a finitely generated, torsion-free nilpotent group, call $G$ its \emph{Mal'cev completion} and call $\ag$ the Lie algebra of $G$. We will denote $\mathrm{gr}(G)$ the simply connected nilpotent Lie group whose Lie algebra is $\mathrm{gr}(\ag)$, we will call this group the \textit{associated Carnot group} (or \textit{Carnotification}) of $G$ or $\G$.

We recall some results on $\textrm L^p$-measure equivalence of nilpotent groups. On one hand, Austin showed that $\mathrm{L}^1$-ME finitely generated groups have bilipschitz asymptotic cones \cite[1.1]{austin2016}. Moreover, Pansu showed that asymptotic cones of a simply connected nilpotent Lie group $G$ are isomorphic to its Carnotification $\mathrm{gr}(G)$ \cite{pansu89-carnot}. Combining these two results, we obtain that two $\mathrm{L}^1$-ME finitely generated nilpotent groups have the same associated Carnot. On the other hand, Delabie, Llosa Isenrich and Tessera showed a (slightly stronger) converse to this result: namely, two virtually nilpotent finitely generated groups with isomorphic Carnot are $\textrm L^p$-OE for some $p>1$ \cite{delabie-llosa-tessera}. We sum up these results in the following statement:

\begin{thm} \cite[1.5]{delabie-llosa-tessera}
    Let $\G$ and $\L$ be two finitely generated virtually nilpotent groups. The following are equivalent: \\
    $(1)$ $\G$ and $\L$ are $\mathrm{L}^1$-OE, \\
    $(2)$ $\G$ and $\L$ are $\mathrm{L}^p$-OE for some $p>1$, \\
    $(3)$ $\G$ and $\L$ have isomorphic associated Carnot groups.
\end{thm}

In view of this result, Corollary \ref{Corollary: estimate on Betti numbers of nilpotent groups} is mostly interesting for distinguishing nilpotent groups with the same associated Carnot. In particular, Corollary \ref{Corollary: estimate on Betti numbers of nilpotent groups} gives finite upper bounds for the number $p$ for which two nilpotent groups with different Betti numbers can be $\mathrm{L}^p$-OE.

Some of our examples will come from central products. Let $G$ and $H$ be two groups with isomorphic centers, let $i : Z(G) \xrightarrow{\simeq} Z(H) $ be such an isomorphism. The central product $G \times_Z H$ is defined as the quotient $(G \times H) / \{ (x, i(x)), x \in Z(G) \}$. This construction is independent of the chosen isomorphism, and yields a Lie group (resp. a Lie algebra) if $G$ and $H$ are Lie groups (resp. Lie algebras). In the case $G$ and $H$ are Lie groups with isomorphic centers and Lie algebras $\ag$ and $\ah$, the central product $G\times_Z H$ is a Lie group with Lie algebra $\ag \times_Z \ah$.

\subsection{Examples of nilpotent groups with different Betti numbers via filiform algebras}

We will describe some examples of nilpotent Lie algebras for which their Lie algebra cohomology is understood. All of the examples we will present come from \cite{shalom2004harmonic}, \cite{garcia-llosa-pallier} and \cite{delabie-llosa-tessera}. More examples are provided in \cite{gotfredsen-kyed}. We could maybe enrich the list of examples by looking at the concrete bounds on filling functions provided in \cite{young-filling-nilpotent}.

We define a family of filiform Lie algebras $\al_n, n\geq 1$, generated by $X_1, \ldots, X_n$, where
\begin{align*}
    [X_i , X_j ] &= (j-i)X_{i+j} \quad \text{ for } i+j \leq n, \\
    &= 0 \qquad \quad \quad \quad \text{ otherwise.}
\end{align*}

Their graded Lie algebras $\mathrm{gr}(\al_n)$ are the so-called universal or model filiform algebras, they are given by generators $Y_1, \ldots, Y_n$, where
\begin{align*}
    [Y_1 , Y_j ] &= (j-1)Y_{j+1} \quad \text{ for } j \leq n -1, \\
    &= 0 \qquad \quad \quad \quad \text{ otherwise.}
\end{align*}

We call $L_n$ and $\mathrm{gr}(L_n)$ the simply connected nilpotent Lie groups integrating $\al_n$ and $\mathrm{gr}(\al_n)$. Let $\L_n < L_n$ and $\L_n' < \mathrm{gr}(L_n)$ be two cocompact lattices inside $ L_n$ and $\mathrm{gr}(L_n)$ respectively. By Mal'cev's theorem \cite[2.18]{raghunathan}, the groups $\L_n$ and $\L_n'$ are finitely generated, nilpotent and torsion-free.

The groups $L_n$ and $\mathrm{gr}(L_n)$ are $\textrm L^p$-OE for some (small) $p>1$ by \cite[1.4]{delabie-llosa-tessera} and hence $\L_n$ and $\L_n'$ are also $\textrm L^p$-OE for the same $p$.
On the other hand, as already pointed out in \cite[p. 152]{shalom2004harmonic}, these Lie algebras satisfy \begin{equation*}
     b_2(\al_n) = 3 < \left\lfloor\frac{1}{2}(n+1)\right\rfloor = b_2(\mathrm{gr}(\al_n)) 
\end{equation*}
for $n \geq 7$. Corollary \ref{Corollary: estimate on Betti numbers of nilpotent groups} shows that $\L_n$ and $\L_n'$ cannot be $\textrm L^p$-OE for $p > n^2 + n +5$ for every $n \geq 7$.

Much less is known for lower values of $n$. For instance, it is not known whether the groups $L_5$ and $\mathrm{gr}(L_5)$ are quasi-isometric or not. This is because all known quasi-isometry invariants (e.g. Betti numbers, homological filling functions, polynomial growth or graded algebras) fail to distinguish these two groups. Nevertheless the central products $\al_5 \times_{Z} \al_3$ and $\mathrm{gr}(\al_5) \times_{Z} \al_3$ have different Betti numbers in degree 3:
\begin{equation*}
    b_3(\al_5 \times_{Z} \al_3) = 8, \quad  b_3(\mathrm{gr}(\al_5) \times_{Z} \al_3) = 9.
\end{equation*}
\cite[Table 2, see the groups $L_{5,7} \times_Z L_{3,2}$ and $L_{5,6} \times_Z L_{3,2}$]{garcia-llosa-pallier}. Corollary \ref{Corollary: estimate on Betti numbers of nilpotent groups} implies that the corresponding groups cannot admit mutually cobounded $\textrm L^p$-ME couplings for large $p>1$. Notice that none of these two central products of Lie algebras is Carnot, but since they have the same graded algebras 
\begin{equation*}
    \mathrm{gr}(\al_5 \times_{Z} \al_3)\simeq \al_5  \times \al_2 \simeq  \mathrm{gr}(\mathrm{gr}(\al_5) \times_{Z} \al_3),
\end{equation*}
the groups $L_5\times_Z L_3$ and $\mathrm{gr}(L_5)\times_Z L_3$ must be $\textrm L^p$-OE for some $p>1$ \cite[1.4]{delabie-llosa-tessera}. It is worth noting that the difference between $L_5\times_Z L_3$ and $\mathrm{gr}(L_5)\times_Z L_3$ is somewhat harder to detect than the difference between $\mathrm{gr}(L_5)\times_Z L_3$ and $\mathrm{gr}(\mathrm{gr}(L_5)\times_Z L_3)$, as in the latter pair the second Betti number differs, while in the former pair this only happens for the third Betti number. In fact it was already known that the simply connected nilpotent Lie groups $\mathrm{gr}(L_5)\times_Z L_3$ and $\mathrm{gr}(\mathrm{gr}(L_5)\times_Z L_3)$ are not $\textrm L^p$-OE for any $p >5$ \cite[1.14]{delabie-llosa-tessera} using cohomology classes in degree 2 satisfying particular conditions when interpreted as central extensions.

We thank Claudio Llosa Isenrich and Gabriel Pallier for communicating these examples to us.

\section{Applications for non-cocompact lattices in rank 1 simple Lie groups}\label{Section: Applications to non-cocomp lattices in rank 1}

The article \cite{bader-sauer} deals with unitary cohomology of semisimple groups and their lattices. One of the main breakthroughs in their work is the use of polynomial cohomology to deal with cohomology of \textit{non-cocompact} lattices. Their methods do not take into account the degrees of the polynomials appearing in their cohomologies, so when dealing with lattices they only treat higher rank lattices, as these are $\textrm L^p$-integrable for every $0<p<\infty$. A non-cocompact lattice $\L$ in a rank 1 simple Lie group $G$ is $\textrm L^p$-integrable for $p < \mathrm{Confdim}(\partial G)$, but not $\textrm L^p$-integrable for $p \geq  \mathrm{Confdim}(\partial G)$. This means that if we want to perform induction as in \cite[Section 6]{bader-sauer}, we need some ideas from Section \ref{Section: A quantitative version of polynomial cohomology}.

We first start with a statement in a more general setting.

\begin{thm}\label{Induction for non cocompact lattice in abstract lcsc group}
    Let $\L$ be a lattice in a locally compact second countable group $G$. Let $(\rho, V)$ be an isometric representation of $\L$ on some separable Banach space $V$. Suppose that for some $d \in \na$, the higher homological filling functions $\mathrm{cFV}_\L^k$ of $\L$ are polynomially bounded for $k \leq d$. Let $N_k = \prod_{i= 2}^k \deg \mathrm{cFV}_\L^i$ and $1 \leq p < \infty$ and suppose that the lattice $\L$ is $\mathrm{L}^{pN_k}$-integrable in $G$. Then there exist continuous surjective linear maps:
    \begin{align*}
        H_\mathrm{ct}^k(G, I^p(\rho)) \twoheadrightarrow H^k(\L, \rho), \\
        \overline{H}_\mathrm{ct}^k(G, I^p(\rho)) \twoheadrightarrow \overline{H}^k(\L, \rho).
    \end{align*}
    
\end{thm}

\begin{proof}
    Our proof uses the same setting as in \cite[Section 6.5]{bader-sauer}.
    Let $X$ be a Borel fundamental domain for the right $\L$-action on $G$ and let $\a : G \times X \to \G$ be its corresponding lattice cocycle (with the convention that $g^{-1}x\a(g, x) \in X$ for $x \in X$ and $g \in G$). We define induction as in \cite[6.23]{bader-sauer}: for $c \in C^k (\L, V)^\L$ and $g_0, \ldots, g_k \in G$, let $I(c)(g_0, \ldots, g_k) : X \to V$ be the measurable function defined by
    \begin{equation*}
        I(c)(g_0, \ldots, g_k) (x) = c(\a(g_0, x), \ldots, \a(g_k, x)).
    \end{equation*}
    Even if we are in a slightly different setting from Proposition \ref{Induction on quantitative polynomial cohomology}, the integrability condition on $\L$ ensures in the same way that the induction map 
    \begin{equation*}
        I: C_{\mathrm{pol}\leq N_k }^k(\L, V)^\L \to L_\mathrm{loc}^p(G^{k+1}, L^p(X, V))^G
    \end{equation*}
    is well-defined. The inclusion $I^p(V) = L^p(X, V)$ into the Fréchet module $L_\mathrm{loc}^p(X, V)$ (with the same action) induces a natural map:
    \begin{equation*}
        L_\mathrm{loc}^p(G^{k+1},  L^p(X, V))^G \to L_\mathrm{loc}^p(G^{k+1}, L_\mathrm{loc}^p(X, V))^G.
    \end{equation*}
    We have a continuous isomorphism $L_\mathrm{loc}^p(G^{k+1}, L_\mathrm{loc}^p(X, V))^G \simeq L_\mathrm{loc}^p(G^{k+1},V)^\L$. We can construct a chain map $f: C(\L^{k+1}, V)^\L \to L_\mathrm{loc}^p(G^{*+1},V)^\L$ by setting for $c \in  C(\L^{k+1}, V)^\L$
    \begin{equation*}
        fc(g_0, \ldots, g_k) = c(\a(g_0, 1_G), \ldots, \a(g_k, 1_G)).
    \end{equation*}
    In \cite[3.5]{blanc}, it is shown that the modules $L_\mathrm{loc}^p(G^{*+1},V)$ form a relatively injective resolution by continuous $\L$-modules of the continuous $\L$-module $V$. Hence, by the fundamental theorem of relative homological algebra, the chain map $f$ admits an inverse chain homotopy equivalence. This means that we can construct a continuous chain homotopy equivalence:
    \begin{equation*}
         T : L_\mathrm{loc}^p(G^{k+1}, L_\mathrm{loc}^p(X, V))^G \simeq L_\mathrm{loc}^p(G^{*+1},V)^\L \to C(\L^{k+1}, V)^\L.
    \end{equation*}
    It can be shown that the composition $T \circ I$ is the comparison map $$C_{\mathrm{pol} \leq N_k}(\L^{k+1}, V)^\L \to C(\L^{k+1}, V)^\L.$$ Hence we obtain the following commutative diagram. 
\begin{center}
    \begin{tikzcd}[scale=0.7em]
	C_{\mathrm{pol} \leq N_k}(\L^{k+1}, V)^\L & L_\mathrm{loc}^p(G^{k+1}, L^p(X, V))^G \\
	C(\L^{k+1}, V)^\L & L_\mathrm{loc}^p(G^{k+1}, L_\mathrm{loc}^p(X, V))^G 
	\arrow["I", from=1-1, to=1-2]
	\arrow[""', from=1-1, to=2-1]
	\arrow[""', from=1-2, to=2-2]
	\arrow["T"', from=2-2, to=2-1]
\end{tikzcd}
\end{center}
Theorem \ref{Quantitative isomorphism of polynomial cohom onto ordinary cohom} says that the comparison map induces a surjective map in cohomology $H_{\mathrm{pol}: N_k \to N_k}^k(\L, V) \to H^k(\L, V)$. We obtain the following commutative diagram:
\begin{center}
    \begin{tikzcd}[scale=0.7em]
	H_{\mathrm{pol}: N_k \to N_k}^k(\L, V) & H^k_\mathrm{ct}(G,  L^p(X, V)) \\
	H^k(\L, V) & H^k_\mathrm{ct}(G, L_\mathrm{loc}^p(X, V)) 
	\arrow["I", from=1-1, to=1-2]
	\arrow[""', twoheadrightarrow, from=1-1, to=2-1]
	\arrow[""', from=1-2, to=2-2]
	\arrow["T"', from=2-2, to=2-1]
\end{tikzcd}
\end{center}
Hence there is a continuous surjective map $H^k_\mathrm{ct}(G,  L^p(X, V)) \to H^k(\L, V)$.

By continuity of $I$ and $T$, we have the same commutative diagram for reduced cohomology. The comparison map $\overline{H}_{\mathrm{pol}: N_k \to N_k}^k(\L, V)  \to \overline{H}^k(\L, V)$ is also surjective by Theorem \ref{Quantitative isomorphism of polynomial cohom onto ordinary cohom}, hence we obtain a continuous surjective map $\overline{H}^k_\mathrm{ct}(G,  L^p(X, V)) \to \overline{H}^k(\L, V)$.
\end{proof}

As a direct consequence of Theorem \ref{Induction for non cocompact lattice in abstract lcsc group} and Shalom's integrability conditions for non-cocompact lattices in rank 1, we obtain Theorem \ref{Intro Theorem: Surjectivity from induced cohomology of rank 1 lattices}.

\begin{proof}[Proof of Theorem \ref{Intro Theorem: Surjectivity from induced cohomology of rank 1 lattices}]

 Let $\L$ be a non-cocompact lattice in a simple Lie group $G$ of rank 1. In \cite[3.7]{shalom2000rigidity-semisimple}, Shalom proved that $\L$ is $\textrm L^q$-integrable in $G$ for every $q < \mathrm{Confdim}(\partial G)$ (though stated more clearly in the Appendix \cite[A.5]{marrakchi-dlSalle}). Hence we can use Theorem \ref{Induction for non cocompact lattice in abstract lcsc group} for any $p < \Cdim (\partial G) / N_k$.
    
\end{proof}

\subsection{Filling functions of non-cocompact lattices in rank 1}

 Leuzinger \cite{leuzinger-rank1} (elaborating on \cite{young-filling-nilpotent}) and Gruber \cite{gruber-filling} obtained sharp bounds for filling functions of non-cocompact lattices of rank 1 Lie groups. These bounds, combined with Theorem \ref{Induction for non cocompact lattice in abstract lcsc group}, will allow us to deduce information about their unitary and Banach cohomology. We are not aware of any precise computations of filling functions for lattices of the quaternionic hyperbolic space of dimension $4n$ in degrees above $n$ nor for the octonionic hyperbolic plane, nevertheless Wenger \cite[7.3]{wenger-asymptotic-rank} still provides some (likely non sharp) upper bounds for filling functions of Carnot groups which imply upper bounds for filling functions of all lattices in rank 1 simple Lie groups.

\begin{thm} \label{Sharp bound for filling functions of rank 1 lattices}
    Let $\L$ be a non-cocompact lattice in $G = \mathrm{Isom}_0 (\mathbb{H}^n_\mathbb{K})$, where $\mathbb{H}^n_\mathbb{K}$ denotes the hyperbolic space of dimension $n$ over $\mathbb{K} = \re, \co$ or the quaternions $\mathbb{H}$. \\
    $\bullet$ \cite[Theorem 3 and 4]{leuzinger-rank1} For $\mathbb{K} = \re$, we have: \\
    $\mathrm{cFV}^{k+1}_\L(t) \sim t^\frac{k+1}{k}$ for $1 \leq k \leq n-2$ \qquad $\mathrm{cFV}^{n}_\L(t) \sim t$.  \\
    $\bullet $ \cite[Theorems 3 and 5]{leuzinger-rank1} For $\mathbb{K} = \co$, we have: \\
     $\mathrm{cFV}^{k+1}_\L(t) \sim t^\frac{k+1}{k}$ for $1 \leq k \leq n-2$, \qquad $\mathrm{cFV}^{n}_\L(t) \sim t^\frac{n+1}{n-1}$, \\
     $\mathrm{cFV}^{k+1}_\L(t) \sim t^\frac{k+2}{k+1}$, for $n \leq k \leq 2n-2$, \qquad $\mathrm{cFV}^{2n}_\L(t) \sim t$. \\
    $\bullet$ \cite[Theorem 3]{leuzinger-rank1}  and \cite[5.9]{gruber-filling} For $\mathbb{K} = \mathbb{H}$, we have: \\
    $\mathrm{cFV}^{k+1}_\L(t) \sim t^\frac{k+1}{k}$ for $1 \leq k \leq n-2$ \qquad $\mathrm{cFV}^{4n}_\L(t) \sim t$.  
\end{thm}

This means that for a non-cocompact lattice $\L$ in $G = \mathrm{Isom}_0 (\mathbb{H}^n_\mathbb{K})$, the constant $N_k = \prod_{j=2}^k \deg (\mathrm{cFV}_\L^j)$ can be computed as: \\
$\bullet$ For $\mathbb{K} = \re$, we have: \\
    $N_k = k$ for $2 \leq k \leq n-1$, \quad $N_n = n-1$.  \\
$\bullet$ For $\mathbb{K} = \co$, we have: \\
     $N_k = k $ for $2 \leq k \leq n-1$, \quad $N_k = k+1$ for $n \leq k \leq 2n-1$, \quad  $N_{2n} = 2n$. \\
$\bullet$ For $\mathbb{K} = \mathbb{H}$, we have: \\
    $N_k = k$ for $2 \leq k \leq n-1$. \\

We now turn to deriving some general upper bounds for the remaining cases.

\begin{prop}\label{Rough upper bound for filling of rank 1 lattices}
    Let $\L$ be a non-cocompact lattice in $G = \mathrm{Isom}_0 (\mathbb{H}^n_\mathbb{K})$, where $\mathbb{H}^n_\mathbb{K}$ denotes the hyperbolic space of dimension $n \geq 2$ over $\mathbb{K} = \re, \co$, the quaternions $\mathbb{H}$ or the octonions $ \mathbb{O}$ (in which case $n=2$).
    For $2 \leq k \leq n-1$, we have:
    \begin{equation*}
        \mathrm{cFV}^{k}_\L(t) \lesssim t^\frac{2^k -1}{2^{k-1} - 1}.
    \end{equation*}
\end{prop}

\begin{proof}
    Let $G = KAN $ denote the Iwasawa decomposition of $G$, where $K$ is a maximal compact subgroup of $G$, $A$ a maximal $\re$-split torus in $G$ and $N$ is the unipotent radical of the minimal parabolic subgroup containing $A$. By the Federer-Fleming deformation technique, the combinatorial filling functions $\mathrm{cFV}_\Lambda^k$ are equivalent to their more classical singular homological variants $\mathrm{FV}_\Lambda^k$ \cite[Theorem 6.6]{bader-sauer}. The group $\L$ acts properly and cocompactly on the space $X_0$ obtained by removing a $\L$-invariant family of horoballs to the hyperbolic space $\mathbb{H}_\mathbb{K}^n = A N$, where the boundary of each horoball (horospheres) is isometric to the nilpotent Lie group $N$. As in the proof of \cite[Theorem 5]{leuzinger-rank1}, we obtain that 
    \begin{equation*}
        \mathrm{FV}_\Lambda^k \sim \mathrm{FV}_N^k
    \end{equation*}
    for every $k \geq 2$.
    The Lie group $N$ is Carnot and at most 2-step nilpotent. Hence \cite[7.3]{wenger-asymptotic-rank} shows that for every $k \geq 2$ we have    
    \begin{equation*}
        \mathrm{FV}^{k}_N(t) \lesssim t^{1 + \frac{2^{k-1}}{1 + 2 +\ldots + 2^{k-2}}} = t^{\frac{1 + 2 +\ldots + 2^{k-1}}{1 + 2 +\ldots + 2^{k-2}}} =t^\frac{2^k -1}{2^{k-1} - 1}.
    \end{equation*}
    which implies the desired result.
    \end{proof}

This result implies that we have the following general upper bound on the constant $N_k$:
\begin{equation*}
    N_k  \leq \prod_{j = 2}^k \frac{2^j -1}{2^{j-1} - 1} = 2^k  -1.
\end{equation*}

\subsection{Consequences for unitary representations and $L^p$-cohomology}

We describe in more detail the implications of this result for particular classes of representations of non-cocompact lattices in simple Lie groups of rank 1 of classical type. We start by dealing with unitary representations.

\begin{cor}
    Let $\L$ be a non-cocompact lattice in $G = \mathrm{Isom}_0 (\mathbb{H}^n_\mathbb{K})$, where $\mathbb{H}^n_\mathbb{K}$ denotes the hyperbolic space of dimension $n \geq 2$ over $\mathbb{K} = \re, \co$, the quaternions $\mathbb{H}$ or the octonions $\mathbb{O}$. Let $(\pi, \mathcal{H})$ be a unitary representation of $\L$. Then if one of the following holds: \\
    $\bullet$ $\mathbb{K} = \re$ and $1 \leq k <  \frac{n-1}{2}$, \\
    $\bullet $ $\mathbb{K} = \co$ or $\mathbb{H}$ and $1 \leq k < n$, \\
    $\bullet $ $\mathbb{K} = \mathbb{O}$, $n = 2$ and $k = 1, 2, 3$, \\
    there exist continuous surjective linear maps:
    \begin{align*}
        H^k_\mathrm{ct}(G, I^2(\pi)) \twoheadrightarrow H^k(\L, \pi), \\
        \overline{H}^k_\mathrm{ct}(G, I^2(\pi)) \twoheadrightarrow \overline{H}^k(\L, \pi).
    \end{align*}
\end{cor}

\begin{proof}
    Let $d = \dim_\re \mathbb{K}$. In order to use Theorem \ref{Intro Theorem: Surjectivity from induced cohomology of rank 1 lattices} for $p = 2$, we need that $2 N_k < d(n+1) -2$.
With our previous computation of $N_k$ obtained after stating Theorem \ref{Sharp bound for filling functions of rank 1 lattices}, this condition reduces to $2 k < d(n+1) -2$ for $1 \leq k < n$. \\ 
For $\mathbb K=\re, d = 1$, this is only satisfied for $1 \leq k <  \frac{n-1}{2}$. \\
For $\mathbb K=\mathbb C, d = 2 \text{ and } \mathbb K=\mathbb H,d=4$ this is satisfied as soon as $1 \leq k < n$. \\
For $\mathbb K=\mathbb O, d =8$ and $n =2$, we use the inequality $N_k \leq 2^k -1$ obtained after Proposition \ref{Rough upper bound for filling of rank 1 lattices} and see that the condition $2N_k < d(n+1) - 2 = 22$ is satisfied for $k =1,  2, 3$.

\end{proof}

\begin{rem}
    The range of $k$ for which this result holds should be larger in the quaternionic case. The reason we cannot go higher is because we do not know good estimates on the filling functions of (non-uniform) lattices in $\mathbb{H}_\mathbb{H}^n$ in degrees between $n$ and $3n$. At least looking at the real and the complex case, we expect to reach all degrees below half of the dimension.
\end{rem}

In \cite[Appendix A]{bader-sauer}, it is shown that the real simple Lie group $G = \mathrm{Isom}_0( \mathbb{H}_\mathbb{O}^2)$ has \emph{property $[T_3]$} \cite[Definition 1.1]{bader-sauer}. The previous result shows that this property is inherited by its lattices.

\begin{cor}\label{Octonionic lattices have T_3}
    Let $\L$ be a non-cocompact lattice in $G = \mathrm{Isom}_0 (\mathbb{H}^2_\mathbb{O})$. Then $\L$ has property $[T_3]$, that is, for every unitary representation $\pi$ of $\L$, we have $H^k (\L, \pi) = 0$ for $k =1, 2, 3$.
\end{cor}

An important Banach representation is the regular representation in $L^p$. Given a locally compact second countable group $G$, we denote by $H^*_\mathrm{ct}(G, L^p(G))$ the (continuous) cohomology of $G$ with values in the right regular representation on $L^p(G) := L^p(G, m_G)$, where $m_G$ is a left Haar measure, we commonly call this space the \textit{(continuous) group $ L^p$-cohomology} of $G$. Some properties that make $ L^p$-cohomology interesting are the following. \\
$\bullet$ It is invariant by quasi-isometries \cite{pansu95}, \cite[1.1]{bourdon-remy-vanishings}. \\
$\bullet$ It has natural geometric interpretations: there are simplicial and de Rham versions of $ L^p$-cohomology, which coincide with the group-theoretic version for well-behaved simplicial complexes and manifolds admitting nice geometric actions \cite[3.2]{bourdon-remy-vanishings}, \cite[A.1]{bourdon-remy-non-vanishing}.

Pansu showed a vanishing result for de Rham $ L^p$-cohomology of manifolds with pinched strictly negative sectional curvature \cite[Théorème A]{pansu08}. This applies in particular to non-compact globally symmetric spaces of rank 1, and by extension to their isometry groups and their cocompact lattices. As an application of Theorem \ref{Induction for non cocompact lattice in abstract lcsc group}, we obtain that non-cocompact lattices satisfy the same vanishing result as their ambient simple Lie groups of rank 1.

\begin{cor}\label{Vanishing of Lp-cohom for rank 1 lattices}
Let $\L$ be a non-cocompact lattice in $G = \mathrm{Isom}_0 (\mathbb{H}^n_\mathbb{K})$, where $\mathbb{H}^n_\mathbb{K}$ denotes the hyperbolic space of dimension $n$ over $\mathbb{K} = \re, \co$ or the quaternions $\mathbb{H}$. Let $k \geq 1$ and $1 \leq p < \infty$. We have $H^k(\L, \ell^p(\L)) = 0$ if \\
$\bullet$ $\mathbb{K} = \re$, $1 \leq k \leq n-2$ and $1 \leq p <  \frac{n-1}{k}$, \\
$\bullet $ $\mathbb{K}= \co$, $1 \leq k < 2n-2$ and $1 \leq p <  \frac{2n + k -1}{2k}$, \\
$\bullet $ $\mathbb{K} =  \mathbb{H}$, $1 \leq k \leq n-1$ and $1 \leq p <  \frac{4n + k -1}{2k}$, \\
$\bullet $ $\mathbb{K} =  \mathbb{O}$, $k = 1, 2, 3$ and $1 \leq p <  \frac{15 + k}{2k}$.
\end{cor}

\begin{proof}
Let $d = \dim_\re \mathbb{K}$. In order to use Theorem \ref{Induction for non cocompact lattice in abstract lcsc group} for some $1 \leq p < \infty$ we need 
\begin{equation*}
    p N_k < \mathrm{Confdim}(\partial \mathbb{H}^n_\mathbb{K}) = d(n+1) -2.
\end{equation*} 
Under this condition there is a surjective continuous map 
\begin{equation*}
    H^k_\mathrm{ct}(G, I^p(\ell^p (\L))) \twoheadrightarrow H^k(\L, \ell^p(\L)).
\end{equation*} 
But the induced representation $I^p(\ell^p (\L))$ acting on $ L^p(X, \ell^p(\L)) \simeq L^p(G)$ is continuously equivalent to the regular representation of $G$ on $L^p(G)$. Cohomology of the latter representation is in turn equivalent to the de Rham $ L^p$-cohomology of $\mathbb{H}^n_\mathbb{K}$. To sum up
\begin{equation*}
      L^pH^k_\mathrm{dR} (\mathbb{H}^n_\mathbb{K}) \simeq H^k_\mathrm{ct}(G, L^p(G)) \simeq  H^k_\mathrm{ct}(G, I^p(\ell^p (\L))) \twoheadrightarrow H^k(\L, \ell^p(\L)).
\end{equation*}
In \cite[Théorème A]{pansu08}, Pansu showed that 
\begin{equation*}
      L^pH^k_\mathrm{dR} (\mathbb{H}^n_\mathbb{K}) = 0 \text{ for } 1 \leq p < 1 + \frac{dn-k-1}{k} \sqrt{-\d},
\end{equation*}
where $\d$ is the pinching of the sectional curvature of $\mathbb{H}^n_\mathbb{K}$. For $\mathbb{K} = \re$, we have $\d = -1$ and for $\mathbb{K} = \co, \mathbb{H}, \mathbb{O}$ we have $\d =- \frac{1}{4}$. Thus, we obtain 
\begin{equation*}
     H^k(\L, \ell^p(\L)) = 0 \text{ for } 1 \leq p < \min \left\{ \frac{ d(n+1) -2}{N_k},  1 + \frac{dn-k-1}{k} \sqrt{-\d} \right\}.
\end{equation*}
Using our computation of $N_k$, case by case verification yields that 
\begin{equation*}
     \frac{ d(n+1) -2}{N_k} \geq   1 + \frac{dn-k-1}{k} \sqrt{-\d}. 
\end{equation*}
Hence we obtain $H^k(\L, \ell^p(\L)) = 0$ for $1 \leq p <  1 + \frac{dn-k-1}{k} \sqrt{-\d} $, which gives the statement after replacing $d$ and $\d$ by their respective values.
\end{proof}

\appendix

\bibliographystyle{amsalpha}
\bibliography{refs.bib}

\noindent Antonio López Neumann \\
Université Paris Cité, Sorbonne Université, CNRS, IMJ-PRG, F-75013 Paris, France \\
lopezneumann@imj-prg.fr \\

\noindent Juan Paucar Zanabria\\
Université Paris-Cité, IMJ-PRG, Paris\\
jpaucar@imj-prg.fr\\

\end{document}